\documentclass[letterpaper,12pt]{article}

\newtheorem{theo}{Theorem}[section]

\newtheorem{lem}[theo]{Lemma}
\newtheorem{defn}[theo]{Definition}

\newtheorem{cor}[theo]{Corollary}

\usepackage{thmtools}
\usepackage{thm-restate}
\newenvironment{proof}{\noindent {\sc Proof}.}
                {\phantom{a} \hfill \framebox[2.2mm]{ } \bigskip}
                
\usepackage[
  top=0.9in,
  bottom=0.9in,
  left=0.9in,
  right=0.9in
]{geometry}

\usepackage{algpseudocode}
\usepackage{float}
\usepackage{algorithm}

\usepackage[colorlinks=true,linkcolor=blue,citecolor=blue]{hyperref}
\usepackage[dvips]{graphics}
\usepackage{amsmath}
\usepackage{amsfonts,enumerate}
\usepackage{pdfpages,caption,geometry}
\usepackage{subcaption}
\usepackage{tikz}
\usepackage{mathrsfs}
\usepackage{boldline,multirow}
\usepackage{array, boldline, makecell, booktabs} 
\usepackage{pdfpages,caption,geometry}
\usepackage{siunitx}
\usepackage{numprint}
\usepackage{subcaption}
\usepackage{mathtools,hhline}
\usepackage{amsmath}
\usepackage{mathtools}
\usepackage{amsfonts,enumerate}
\usepackage{titling}
\usepackage{blindtext}
\usepackage{mathtools}

\DeclarePairedDelimiter\floor{\lfloor}{\rfloor}
\usepackage{hyperref}
\hypersetup{
    colorlinks=true,
    linkcolor=blue,
    filecolor=magenta,      
    urlcolor=cyan,
    pdftitle={Overleaf Example},
    pdfpagemode=FullScreen,
    }

\newcommand{\ZZ}{\mathbb{Z}}

\newcommand{\D}{\mathcal{D}}
\newcommand{\F}{\mathcal{F}}

\usepackage{eurosym,datatool,longtable, color,colortbl}
\usepackage{lscape}
\definecolor{RC}{rgb}{0.9,1,1}
\definecolor{Gray}{gray}{0.9}

\definecolor{PINK}{cmyk}{0, 0.2, 0.001, 0.001}

\definecolor{ao(english)}{rgb}{0.0,0.4,0.0}
\newcolumntype{g}{>{\columncolor{Gray}}c}
\newcolumntype{g}{>{\columncolor{Gray}}c}

\title{On the Generalized Honeymoon Oberwolfach Problem}

\author{Masoomeh Akbari\thanks{Department of Mathematics, University of Ottawa, ON, Canada.
Email: makba074@uottawa.ca}}

\begin{document}
\maketitle \baselineskip 18pt

\begin{abstract}

The generalized Honeymoon Oberwolfach Problem (HOP) asks whether it is possible to seat $2n$ participants consisting of $n$ newlywed couples at a conference with $s$ tables of size $2$ and $t$ ``round'' tables of sizes $2m_1, 2m_2, \ldots, 2m_t$, where  \(n = s + \sum_{i=1}^{t} m_i \) with all $m_i \geq 2$, over several nights so that each participant sits next to their spouse every time and next to each other participant exactly once. We denote this problem by $\mathrm{HOP}(2^{\langle s \rangle}, 2m_1, \ldots, 2m_t)$.

This paper is the first of two papers investigating the generalized HOP.  While the second paper will deal with the generalized HOP with a single {\em round} table (i.e. table of size at least $4$),  the present work develops solutions for the generalized HOP with multiple round tables.  In particular, we present solutions to certain cases with two round tables, showing  that a solution to $\mathrm{HOP}(2^{\langle s \rangle}, 2m_1, 2m_2)$ exists when  $n \equiv 1 \pmod{(2m_1 + 2m_2)}$ or $n \equiv m_1 + m_2 \pmod{(2m_1 + 2m_2)}$.  We also develop solutions for cases with small round tables, showing that  $\mathrm{HOP}(2^{\langle s \rangle}, 2m_1, \dots, 2m_t)$ has a solution whenever $m = m_1 + \dots + m_t \leq 10$,  $n = s + m$ is odd, and $n(n - 1) \equiv 0 \pmod{2m}$.

\end{abstract}

\section{Introduction}

The well-known Oberwolfach Problem (OP) was posed by Ringel at a conference in Oberwolfach in 1967. It asks whether it is possible to seat $n$ participants in a room with $t$ round tables of sizes $m_1, m_2, \ldots, m_t$ over several nights so that all tables are full at each meal and each participant sits beside every other participant exactly once.

In this paper, we focus on a variant of the OP called the Honeymoon Oberwolfach Problem (HOP), which is due to \v{S}ajna~\cite{LDMSaj}. 
In this variation, the problem is formulated as the question whether it is possible to seat $2m_1 + 2m_2 + \ldots + 2m_t = 2n$ participants, consisting of $n$ newlywed couples, at $t$ round tables of sizes $2m_1, 2m_2, \ldots, 2m_t$ for $2n - 2$ nights, so that each participant sits next to their spouse every night and next to every other participant exactly once. 

In the language of graph theory, a solution to the HOP corresponds to a decomposition of $K_{2n} + (2n - 3)I$ into $2$-factors. Here, the multigraph $K_{2n} + (2n - 3)I$ is obtained from the complete graph $K_{2n}$ by adjoining $2n-3$ additional copies of a fixed $1$-factor $I$, and the $2$-factors in the decomposition are vertex-disjoint unions of cycles of lengths $2m_1, 2m_2, \ldots, 2m_t$, where within each cycle, every other edge is a copy of an edge of $I$. This problem is denoted by HOP$(2m_1, 2m_2, \ldots, 2m_t)$. In the special case where $m_1 = m_2 = \ldots = m_t$ and $n = tm$, it is denoted by HOP$(2n; 2m)$. HOP has been studied by Jerade, Lepine and  Šajna, and some significant cases of it have been solved \cite{MRMSaj, LDMSaj}.

In {\cite {LDMSaj}}, the HOP was defined only for tables of size at least $4$. In this paper, we generalize the problem to include tables of size $2$, leading to a new version of the problem that has not been previously studied.
The generalized HOP preserves the original seating conditions of the HOP, except that the $2n$ participants are seated at $s$ tables of size $2$ and $t$ tables of sizes $2m_1, 2m_2, \ldots, 2m_t$, under the assumption that $n = s + m_1 + \ldots + m_t$ and all $m_i \geq 2$. We denote this problem by HOP$(2^{\langle s \rangle}, 2m_1, 2m_2, \ldots, 2m_t)$, and refer to tables of size at least $4$ as \emph{round tables}.

Our goal is to determine whether the obvious necessary conditions for HOP$(2^{\langle s \rangle}, \allowbreak 2m_1, \ldots, 2m_t)$ to have a solution are also sufficient. As an initial step toward this problem, we prove the following results in this paper.

\begin{restatable}{theo}{mainresultOne}{\label{theo-new-1}}
Let $s, m_1, m_2$ be integers such that $s \geq 0$ and $2 \leq m_1 \leq m_2$. Let $m = m_1 + m_2$ and $n = s + m$. Then {\rm HOP}$(2^{\langle s \rangle}, 2m_1, 2m_2)$ has a solution in each of the following cases:

\begin{enumerate}[\bf(i)]
    \item  $n \equiv 1 \pmod{2m}$
    \item  $m$ is odd and $n \equiv m \pmod{2m}$
\end{enumerate}
\end{restatable}

\begin{restatable}{theo}{mainresultTwo}{\label{theo:HOP(Cm1,...,Cmt)-dec}}
Let $s\geq 0$, and $2\leq m_1\leq\ldots\leq m_t$ be integers. Assume $m=m_1+\ldots+m_t\leq 10$, and that $n=s+m$ is an odd integer such that $n(n-1)\equiv 0\ ({\rm mod}\ 2m)$. Then HOP$(2^{\langle s \rangle},2m_1, \ldots, 2m_t)$ has a solution. 
\end{restatable}

This paper is structured as follows. In Section~\ref{sec:2}, we begin by defining the necessary terminology. Then, in Section~\ref{sec:3}, we adapt the approaches used for the HOP in~\cite{LDMSaj} to its generalized form, and present the previous results needed to prove our main theorems. After discussing the techniques and tools needed to address the generalized HOP in Section~\ref{sec:4}, we conclude in Sections~\ref{sec:5}  and \ref{sec:6} with the proofs of our two main results.


\section{Terminology}{\label{sec:2}}

Graphs in this paper will be loopless, but may contain parallel edges or directed edges. For any simple graph \( G \), we use \( \lambda G \) to denote the multigraph obtained by replacing each edge of \( G \) with \( \lambda \) parallel copies. The disjoint union of graphs $G_1$ and $G_2$ is denoted by $G_1 \,\dot{\cup}\, G_2$. 
 As usual, \( K_n \) denotes the complete graph on \( n \) vertices, so \( \lambda K_n \) denotes the \( \lambda \)-fold complete graph of order $n$. We use the symbol \( K_{m[k]} \) to denote the complete equipartite graph with \( m \) parts of size \( k \). 

A collection of subgraphs $\{H_1, H_2, \ldots, H_t\}$ of a graph $G$ is called a \textit{decomposition} of $G$ if the set of their edge sets $\{E(H_1), E(H_2), \ldots, E(H_t)\}$ forms a partition of $E(G)$; if this is the case, we write $G = H_1 \oplus H_2 \oplus \dots \oplus H_t$. Furthermore, if all subgraphs $H_1, H_2, \ldots, H_t$ are isomorphic to a graph $H$, the collection is referred to as an \textit{$H$-decomposition} of $G$.

A graph is called \textit{$r$-regular} if all of its vertices have degree $r$. A  {\em $(C_{m_1}, C_{m_2},\ldots,C_{m_t})$-subgraph} of a graph $G$ is a $2$-regular subgraph of $G$ consisting of $t$ disjoint cycles of lengths $m_1,m_2, \ldots,m_t$;
if this subgraph is spanning, it is called a \textit{$(C_{m_1}, \dots, C_{m_t})$-factor}. A  {\em $(C_{m_1}, C_{m_2},\allowbreak \ldots,C_{m_t})$-decomposition} of a graph $G$ is a decomposition of $G$ into $(C_{m_1}, C_{m_2},\ldots,C_{m_t})$-subgraphs; if these subgraphs are factors, the decomposition is called a \textit{$(C_{m_1}, \dots, C_{m_t})$-factorization}. In the case of uniform cycle lengths, that is,  $m_1 = \dots = m_t = m$, we use the terms \textit{$(C_m^{\langle t \rangle})$-subgraph}, \textit{$C_m$-factor}, \textit{$(C_m^{\langle t \rangle})$-decomposition}, and \textit{$C_m$-factorization}, respectively. Note that the symbol $C_m^{\langle t \rangle}$ represents the $t$-tuple $(C_{m}, C_{m},\ldots,C_{m})$. A {\em $(K_2^{\langle s \rangle},C_{m_1}, \ldots,C_{m_t})$-factor} of a graph $G$ is a spanning subgraph of $G$ that is a disjoint union of $s$ copies of $K_2$ and a $(C_{m_1},\ldots,C_{m_t})$-subgraph of $G$.  A {\em $(K_2^{\langle s \rangle},C_{m_1}, \ldots,C_{m_t})$-factorization} of $G$ is a decomposition of $G$ into $(K_2^{\langle s \rangle},C_{m_1}, \ldots, C_{m_t})$-factors.

Let $I$ be a 1-factor in $K_{2n}$. An edge of $K_{2n}$ which belongs to $E(I)$ is called an {\em $I$-edge}; all other edges are {\em non‐I‐edges}. A graph $K_{2n}$  with  all  $I$-edges deleted is denoted by $K_{2n}- I$, and a graph $K_{2n}$  with $\lambda$ additional copies of each   $I$-edge is denoted by $K_{2n}+\lambda I$. Note that  additional copies of $I$-edges are also considered  $I$-edges in the graph $K_{2n}+\lambda I$. 
A cycle $C$ in $K_{2n} + \lambda I$ is called an {\em $I$-alternating} cycle if the $I$-edges and non-$I$-edges alternate along $C$. 
Let $F$ be a  $(C_{m_1}, C_{m_2},\ldots,C_{m_t})$-subgraph of  $K_{2n} + \lambda I$. If every cycle in  $F$ is $I$-alternating, then $F$ is said to be {\em $I$-alternating}.
%
A $(K_2^{\langle s \rangle},C_{m_1}, \ldots,\allowbreak C_{m_t})$-factor of $K_{2n} + \lambda I$  is called  {\em $I$-alternating} if  its $(C_{m_1},\ldots, C_{m_t})$-subgraph is $I$‐alternating, and all other edges are $I$-edges. Moreover, a $(K_2^{\langle s \rangle},C_{m_1}, \ldots,C_{m_t})$-factorization is {\em $I$-alternating} if all of its  $(K_2^{\langle s \rangle},C_{m_1}, \ldots,C_{m_t})$-factors are $I$-alternating.

Let $G_1$ and $G_2$ be two vertex-disjoint simple graphs. The \emph{join} of $G_1$ and $G_2$, denoted by $G_1 \bowtie G_2$, is a simple graph consisting of the union of graphs $G_1$ and $G_2$ along with all edges with one end in $G_1$ and the other in $G_2$.

The \emph{circulant graph} $\mathrm{Circ}(n; S)$, where $S \subseteq \mathbb{Z}_n^\ast$ and $S = -S$,  is a graph with vertex set $\{x_i : i \in \mathbb{Z}_n\}$ and edge set $\{x_i x_{i+d} : i \in \mathbb{Z}_n,\ d \in S\}$. An edge of the form $x_i x_{i+d}$ is said to have \emph{difference} $d$. Since each edge of difference $d$ is also of difference $n - d$, we may assume that all differences lie in the set $\{1, 2, \ldots, \lfloor \tfrac{n}{2} \rfloor\}$. When \( n \) is even, the edge \( x_i x_{i + \tfrac{n}{2}} \) connects a pair of vertices that are ``diametrically opposite''; we call $d=\tfrac{n}{2}$ a \emph{diameter difference}.
In many of the constructions given in this paper, the complete graph $K_{n}$ is viewed as $\mathrm{Circ}(n-1; \pm S) \bowtie K_1$, where $S = \{1, 2, \ldots, \lfloor\tfrac{n-1}{2} \rfloor\}$ and the vertex of $K_1$ is denoted by $x_{\infty}$. Note that an edge of the form $x_ix_{\infty}$ is said to be of {\em difference infinity}.

Let $G = \mathrm{Circ}(n; S)$ be a circulant graph with vertex set $\{x_0, x_1, \ldots, x_{n-1}\}$, where the vertices are cyclically ordered according to increasing subscripts. We define interval notation on this cyclic ordering as follows: the interval $[x_i, x_j] = \{x_i, x_{i+1}, \ldots, x_j\}$, $(x_i, x_j] = \{x_{i+1}, x_{i+2}, \ldots, x_j\}$, $[x_i, x_j) = \{x_i, x_{i+1}, \ldots, x_{j-1}\}$, and $(x_i, x_j) = \{x_{i+1}, x_{i+2}, \ldots, x_{j-1}\}$. Note that all subscripts are evaluated modulo $n$.


\section{The Main Approach and Previous Results}{\label{sec:3}}

We would like to model the generalized HOP using an appropriate graph derived from $K_{2n}$. Let $I$ denote the $1$-factor of $K_{2n}$ corresponding to the $n$ couples. Then, a solution to ${\rm HOP}(2^{\langle s \rangle}, 2m_1, \allowbreak \ldots, 2m_t)$ is equivalent to an $I$-alternating $(K_2^{\langle s \rangle}, C_{2m_1}, \ldots, C_{2m_t})$-factorization of $K_{2n} + (\gamma - 1)I$ for an appropriate $\gamma$. Note that  the number of non-$I$-edges in each $(K_2^{\langle s \rangle}, C_{2m_1}, \ldots, C_{2m_t})$-factor, that is $\frac{1}{2} \sum_{i=1}^t 2m_i$, must divide the total number of non-$I$-edges in $K_{2n} + (\gamma - 1)I$, which is $\frac{2n(2n - 2)}{2}$.  Thus, the obvious necessary condition for ${\rm HOP}(2^{\langle s \rangle}, 2m_1, \ldots, 2m_t)$ to have a solution is $\left(\sum_{i=1}^{t} m_i\right) \mid 2n(n - 1)$. In this case, $\frac{2n(n - 1)}{\sum_{i=1}^{t} m_i}$ must be an integer representing the exact number of $(K_2^{\langle s \rangle}, C_{2m_1}, \ldots, C_{2m_t})$-factors in the required decomposition; that is to say, $\gamma = \frac{2n(n - 1)}{\sum_{i=1}^{t} m_i}$.

Having established the necessary conditions, we now move to the primary tools for our construction.

\begin{defn}{\rm{\cite{LDMSaj}}} {\rm 
Let $G$ be a simple graph. An {\em HOP‐coloring‐orientation} of $4G$ is a 3‐edge‐coloring of $4G$ with colors blue, pink, and black, together with an orientation of the black edges such that for any two adjacent vertices in $4G$, among the four parallel  edges between them, one edge is blue, one edge is pink, and two of the edges are black  with the opposite orientations. The multigraph $4G$ with a given HOP‐coloring‐orientation is denoted by $4G^{\bullet}.$
}
\end{defn}

\begin{defn}{\rm{\cite{LDMSaj}}} {\label{def}}{\rm
Let $G$ be a simple graph and  $\D$  a decomposition of $4G^{\bullet}$ into $2$-regular subgraphs. We say that $\D$ is {\em HOP} if it satisfies the following condition.
\begin{description}
 \item [(C1)] For every cycle $C$ in $\D$, any two adjacent edges  of $C$ satisfy one of the following: 
\begin{itemize}
\item one is blue and the other pink; 
\item one is blue and the other black with an orientation toward the blue edge; 
\item one is pink and the other black with an orientation away from the pink edge; 
\item both are black and oriented in the same way. 
\end{itemize}
\end{description}
}
\end{defn}

The following theorem which allow us to convert the problem from the multigraph $K_{2n} + (\gamma - 1)I$ to the multigraph $4K_n^{\bullet}$  is a generalization of  \cite[Theorem 4.3]{LDMSaj}.

\begin{theo}{\label{theo:Gtool1}}
Let $s\geq 0$ and $2\leq m_1\leq \ldots\leq m_t$ be integers. Let $n=s+m_1+m_2+\ldots +m_t$. Then HOP$(2^{\langle s \rangle},2m_1, 2m_2,\ldots, 2m_t)$ has a solution if and only if $4K_n^{\bullet}$ admits an HOP $(C_{m_1}, C_{m_2},\ldots,C_{m_t})$-decomposition.
\end{theo}
\begin{proof} 
The proof is very similar to the proof of \cite[Theorem 4.3]{LDMSaj}. Here, we show only the part of the proof that needs to be adjusted.

First, assume HOP$(2^{\langle s \rangle}, 2m_1, 2m_2, \ldots, 2m_t)$ has a solution. Then $K_{2n} + (\gamma - 1)I$ admits an $I$-alternating $(K_2^{\langle s \rangle}, C_{2m_1}, \ldots, C_{2m_t})$-factorization for $\gamma = \frac{2n(n - 1)}{\sum_{i=1}^{t} m_i}$; call this factorization \(\F\). 
Let $V(K_{2n} + (\gamma - 1)I) = \{x_i^\alpha : i \in \ZZ_n, \alpha \in \ZZ_2\}$, where $x_i^0x_i^1 \in E(I)$ for all $i \in \ZZ_n$. Let  $4K_n^{\bullet}$ be $V(4K_n^{\bullet}) = \{x_i : i \in \ZZ_n\}$.
Assign to each non-$I$-edge of $K_{2n}+(\gamma-1)I$ an edge of $4K_n^{\bullet}$ as follows. For $i,j\in \ZZ_n$ with $i\neq j$, assign the blue copy of $x_ix_j$ to the edge  $x_i^0 x_j^0$;  the pink  copy of $x_ix_j$ to the edge  $x_i^1 x_j^1$;  the black arc $(x_i, x_j)$  to the edge  $x_i^0 x_j^1$; and the black arc $(x_j, x_i)$  to the edge  $x_i^1 x_j^0$.
This mapping, call it $\phi$, is a bijection from the set of non-$I$-edges of $K_{2n} + (\gamma - 1)I$ to the edge set of $4K_n^{\bullet}$.
Let \(C\) be a $2m$-cycle in \(\F\). By contracting the $I$-edges in the cycle \(C\) and applying the bijection \(\phi\), we obtain a cycle \(C'\) in $4K_n^{\bullet}$
of length \(m\). 
Thus, $\phi$ can be viewed as mapping each $I$-alternating $(K_2^{\langle s \rangle}, C_{2m_1}, \ldots, C_{2m_t})$-factor in $\F$ to a $(C_{m_1}, \ldots, C_{m_t})$-subgraph in $4K_n^{\bullet}$; hence, it maps the $I$-alternating $(K_2^{\langle s \rangle}, \allowbreak C_{2m_1}, \ldots, C_{2m_t})$-factorization $\F$ to a $(C_{m_1}, \ldots, C_{m_t})$-decomposition $\D$ of $4K_n^{\bullet}$. Notice that the copies of $K_2$ in the $(K_2^{\langle s \rangle}, C_{2m_1}, \ldots, C_{2m_t})$-factors ``disappear'' under this mapping; in other words, they are irrelevant. 
Analogous to the proof of \cite[Theorem 4.3]{LDMSaj}, it can be shown that any two adjacent edges in $C'$ satisfy Condition {\bf (C1)} of Definition \ref{def}. Therefore, $\mathcal{D}$ is an HOP $(C_{m_1}, C_{m_2}, \ldots, C_{m_t})$-decomposition of $4K_n^{\bullet}$.

Conversely, assume $4K_n^{\bullet}$ admits an HOP $(C_{m_1}, C_{m_2},\ldots,C_{m_t})$-decomposition  $\D$. Thus, the number of edges of each $(C_{m_1}, C_{m_2},\ldots,C_{m_t})$-subgraph in $\D$ must divide the number of edges of $4K_n^{\bullet}$, that is,   $\big(\sum_{i=1}^{t} m_i\big)|2n(n-1)$, and $\gamma=\frac{2n(n -1)}{\sum_{i=1}^{t} m_i}$ is an integer. Similarly to the proof of \cite[Theorem 4.3]{LDMSaj}, we can show that a $(C_{m_1}, C_{m_2},\ldots,C_{m_t})$-subgraph of $4K_n^{\bullet}$ ``lifts'' to an $I$‐alternating $(C_{2m_1}, C_{2m_2},\ldots,C_{2m_t})$‐subgraph \(F\) of $K_{2n}+(\gamma-1)I$. We extend this subgraph  to an $I$-alternating $(K_2^{\langle s \rangle},C_{2m_1}, \ldots,C_{2m_t})$-factor \(F'\) of $K_{2n}+(\gamma-1)I$ as follows.
 Since $2n= 2s+2m_1+ 2m_2+\ldots+ 2m_t$, we know there are $2s$ vertices in the vertex set of $K_{2n}+(\gamma-1)I$ that are not in \(F\). Moreover, since the $t$ cycles are $I$-alternating, we can find $s$ pairs among the remaining $2s$ vertices such that each pair forms an $I$-edge.  Thus, any $I$‐alternating $(C_{2m_1}, C_{2m_2},\ldots,C_{2m_t})$‐subgraph extends uniquely to an $I$-alternating $(K_2^{\langle s \rangle},C_{2m_1}, \ldots,C_{2m_t})$-factor. 
 Therefore,  any $(C_{m_1}, C_{m_2},\ldots,C_{m_t})$-subgraph in $\D$ ``lifts'' to an $I$-alternating $(K_2^{\langle s \rangle},C_{2m_1}, \ldots,C_{2m_t})$-factor in $K_{2n}+(\gamma-1)I$.
 Since $|\D|=\gamma$, we have $\gamma$ $I$-alternating $(K_2^{\langle s \rangle},C_{2m_1}, \ldots,C_{2m_t})$-factors which form an $I$-alternating $(K_2^{\langle s \rangle},C_{2m_1}, \ldots,C_{2m_t})$-factorization of $K_{2n}+(\gamma-1)I$; this is equivalent to a solution for HOP$(2^{\langle s \rangle},2m_1, 2m_2,\ldots, 2m_t)$. 
\end{proof}

Thus, by Theorem~\ref{theo:Gtool1},  {\rm HOP}$(2^{\langle s \rangle}, 2m_1, 2m_2, \ldots, 2m_t)$ admits a solution if and only if $4K_n^{\bullet}$ admits an HOP $(C_{m_1}, C_{m_2}, \ldots, C_{m_t})$-decomposition. We henceforth focus on finding HOP $(C_{m_1}, C_{m_2}, \ldots, C_{m_t})$-decompositions of $4K_n^{\bullet}$.  Below, we introduce tools and previous results that allow us to use decompositions of simpler graphs, such as $K_n$ and $2K_n$, and extend them to HOP decompositions of $4K_n^{\bullet}$.

In the next lemma, which is a generalization of \cite[Lemma 4.5]{LDMSaj}, the symbol $2G^{\circ}$ represents the multigraph $2G$ with a 2‐edge‐coloring with colors pink and black such that for any two vertices in $2G$, the two parallel edges between them have colors  pink and black.

\begin{lem}{\rm{\cite{LDMSaj}}}{\label{lem:Gtool2}}
Assume that $2G^{\circ}$ admits a $(C_{m_1}, C_{m_2},\ldots,C_{m_t})$-decomposition $\F$ with the property that every $m$-cycle of $\F$, for $m\geq 3$, contains an even number of pink edges. Then $4G^{\bullet}$ admits an HOP $(C_{m_1}, C_{m_2},\ldots,C_{m_t})$-decomposition.
\end{lem}

The following two lemmas are straightforward generalizations of \cite[Lemma 8.1]{LDMSaj} and \cite[Lemma 8.2]{LDMSaj}, respectively.

\begin{lem}{\rm{\cite{LDMSaj}}}{\label{Gtool3}}
Let $G$ be a simple graph. Let $H_1, H_2, \ldots, H_s$ be subgraphs of $G$ such that $G = H_1 \oplus H_2 \oplus \dots \oplus H_s$. If, for all $i \in \{1, \ldots, s\}$, the multigraph $4H_i^{\bullet}$ admits an HOP decomposition $\D_i$ into $(C_{m_1}, \ldots, C_{m_t})$-subgraphs, then $\D = \bigcup_{i=1}^{s} \D_i$ is an HOP decomposition of $4G^{\bullet}$ into $(C_{m_1}, \ldots, C_{m_t})$-subgraphs.
\end{lem}

\begin{lem}{\rm{\cite{LDMSaj}}}{\label{lem:Gtool4}}
Let $G$ be a simple graph. If $G$ admits a $(C_{m_1}, C_{m_2},\ldots,C_{m_t})$-decomposition, then $4G^{\bullet}$ admits an HOP  $(C_{m_1}, C_{m_2},\ldots,C_{m_t})$-decomposition.
\end{lem}


Lemma \ref{lems-1starter} summarizes a set of results from~\cite{LDMSaj}; these, along with the well-known results on decompositions of $K_n$ presented in Theorems \ref{theo:(Cm1,Cm2)-dec}, \ref{theo:Cm-dec}, and \ref{theo:OP4}, are needed for the proof of our main theorems. 

\begin{lem}{\rm{\cite {LDMSaj}}}{\label{lems-1starter}}
The following HOP $2$-factorizations exist: 
\begin{enumerate}[\bf(i)]
\item  {\label{lems-3starter-2}} an HOP $(C_2, C_2, C_5)$- and $(C_4, C_5)$-factorization of $4K_9^{\bullet}$; 
\item {\label{lems-3starter-4}} an HOP $(C_2, C_m)$-factorization of $4K_{m+2}^{\bullet}$, for $m\geq 5$ and $m \equiv 1\ ({\rm mod}\ 4)$;
\item {\label{lems-1starter-4}} an HOP $(C_2, C_m)$-factorization of $4K_{m+2}^{\bullet}$, for $m\geq 3$ and $m \equiv 3\ ({\rm mod}\ 4)$.
\end{enumerate}
\end{lem}


\begin{theo}{\rm{\cite { AbGrace, ADGAV, Blinc02, 3partite, El-Zanati}}}{\label{theo:(Cm1,Cm2)-dec}}
Let \( 3 \leq m_1 \leq m_2 \) and $m=m_1+m_2$. Then there exists a $(C_{m_1}, C_{m_2})$-decomposition of $K_n$ in each of the following two cases:
\begin{enumerate}[\bf(i)]
\item  $n\equiv 1\ ({\rm mod}\ 2m)$;
\item  $m$ is odd, $n\equiv m\ ({\rm mod}\ 2m)$, and $(m_1,m_2,n)\neq(4,5,9)$.
\end{enumerate}
\end{theo}

\begin{theo}{\rm{\cite {BAHG,cycleIII}}}{\label{theo:Cm-dec}}
Let \( 3 \leq m \leq n \) be integers, and let \( n \) be odd. Then \( K_n \) admits a \( (C_m) \)-decomposition if and only if \( m \) divides the number of edges in \( K_n \).
\end{theo}

\begin{theo}{\rm{\cite{TT-21}}}{\label{theo:OP4}}
Let \( 3 \leq m_1 \leq m_2 \) be integers, and let \( n=m_1+m_2 \) be odd. Then \( K_n \) admits a \( (C_{m_1}, C_{m_2}) \)-factorization if and only if \( (m_1, m_2) \notin \{(3,3), (4,5)\} \).
\end{theo}

\begin{theo}{\rm{\cite{ADGAV}}}\label{theo:(Cm1,...,Cmt)-dec}
Assume $3 \leq m_1 \leq \ldots \leq m_t$ are integers and $m_1 + \ldots + m_t = m \leq 10$. There exists a $(C_{m_1}, \ldots, C_{m_t})$-decomposition of $K_n$ if and only if $m \leq n$, $n$ is odd, $m$ divides the number of edges in $K_n$, and $n \neq 9$ whenever $(m_1, m_2, \ldots, m_t) \in \{(3, 3), (4, 5)\}$.
\end{theo}


\section{The Tools}{\label{sec:4}}


\begin{lem}{\label{lem:Newtool-00}}
Assume $G$ is a simple graph. Let $H$ be a $(C_{m_1}, C_{m_2}, \ldots, C_{m_t})$-subgraph of $G$, and let $H'$ be an edge-disjoint union of two $1$-regular subgraphs of $G$ of order $2\alpha$.
If $G$ admits a decomposition into subgraphs isomorphic to $H \mathbin{\dot{\cup}} H'$, then $4G^{\bullet}$ admits an HOP $(C_2^{\langle \alpha \rangle}, C_{m_1}, C_{m_2}, \ldots, \allowbreak C_{m_t})$-decomposition.
\end{lem}

\begin{proof}
Let $\D= \{F_1, F_2, \ldots, F_k\}$ be a decomposition of $G$ such that each $F_i$, for $i\in\{1,2,\ldots,k\}$, is isomorphic to $H\mathbin{\dot{\cup}} H'$. 
Fix $i\in\{1,2,\ldots,k\}$; then $F_i=H_i \mathbin{\dot{\cup}} H'_i$,  where $H_i$ is isomorphic to $H$, and $H'_i$ is isomorphic to $H'$. Now, let $H_i^{(1)},H_i^{(2)},H_i^{(3)},H_i^{(4)}$ be four copies of  the 2-regular subgraph $H_i$. 

Next, for $j\in\{1,2,\ldots,2\alpha\}$,  let $E_j$ be the $j^{{\rm th}}$ subgraph isomorphic to $K_2$ in $H'_i$. We may assume $E_1, \ldots, E_{\alpha}$ as well as $E_{\alpha+1}, \ldots, E_{2\alpha}$ are pairwise  disjoint, and $E_1\cup \ldots \cup E_{\alpha}$ is edge-disjoint from $E_{\alpha+1} \cup \ldots \cup E_{2\alpha}$.  For each $j\in\{1,2,\ldots,2\alpha\}$, let $E_j^{(1)}$ and $E_j^{(2)}$ be 2-cycles  that jointly contain four (parallel) copies of the edge of $E_j$. Now, apply Lemma \ref{lem:Gtool4} to color the edges in $H_i^{(1)},H_i^{(2)},H_i^{(3)},H_i^{(4)}$ and $E_j^{(1)}, E_j^{(2)}$, for all $j$,  so that they satisfy Condition {\bf (C1)} of Definition \ref{def}, and so that
\begin{itemize}
\item $H_i^{(1)}\oplus H_i^{(2)}\oplus H_i^{(3)}\oplus H_i^{(4)}=4H_i^{\bullet}$, and
\item $E_j^{(1)}\oplus E_j^{(2)}=4E_j^{\bullet}$.
\end{itemize}
Now, let $T_{i_1}, T_{i_2}, T_{i_3},T_{i_4}$ be the following subgraphs:
 $$T_{i_1}= H_i^{(1)}\cup  E_1^{(1)} \cup E_2^{(1)} \cup \ldots \cup E_{\alpha}^{(1)};$$
 $$T_{i_2}= H_i^{(2)}\cup  E_{\alpha+1}^{(1)} \cup E_{\alpha+2}^{(1)} \cup \ldots \cup E_{2\alpha}^{(1)};$$
 $$T_{i_3}= H_i^{(3)}\cup  E_1^{(2)} \cup E_2^{(2)} \cup \ldots \cup E_{\alpha}^{(2)};$$
 $$T_{i_4}= H_i^{(4)}\cup  E_{\alpha+1}^{(2)} \cup E_{\alpha+2}^{(2)} \cup \ldots \cup E_{2\alpha}^{(2)}.$$
It is easy to see that  $T_{i_1}, T_{i_2}, T_{i_3},T_{i_4}$ are all $(C_2^{\langle \alpha \rangle}, C_{m_1}, \ldots, C_{m_t})$-subgraphs of $4G^{\bullet}$ that satisfy Condition {\bf (C1)} of Definition \ref{def}, and  
$$T_{i_1}\oplus T_{i_2} \oplus T_{i_3}\oplus T_{i_4}=4(H_i \mathbin{\dot{\cup}}  H_i')$$
We apply the above procedure for each $F_i$, for $i\in\{1,2,\ldots,k\}$,  so each $F_i$ gives rise to four $(C_2^{\langle \alpha \rangle}, C_{m_1}, \ldots, C_{m_t})$-subgraphs. Since $\D= \{F_1, F_2, \ldots, F_k\}$ is a decomposition of $G$, the resulting $(C_2^{\langle \alpha \rangle}, C_{m_1}, \ldots, C_{m_t})$-subgraphs form an HOP $(C_2^{\langle \alpha \rangle}, C_{m_1}, \ldots, C_{m_t})$-decomposition of $4G^{\bullet}$. 
\end{proof}

\begin{cor}{\label{lem:Newtool-01}}
Let $m_1, m_2, \ldots, m_t \geq 3$ be integers. Assume that $G$ is a simple graph that admits a $(C_{m_1}, C_{m_2}, \ldots, C_{m_t})$-decomposition, where $m_1, m_2, \ldots, m_s$ are even for some $s \leq t$.
Then $4G^{\bullet}$ admits an HOP $(C_2^{\langle \frac{m_1+m_2+\ldots+m_s}{2} \rangle}, C_{m_{s+1}},\ldots, C_{m_t})$-decomposition.  
\end{cor}
\begin{proof}
Since a cycle of even length $m$ is an edge-disjoint union of two $1$-regular graphs of order $m$,  we see that $C_{m_1} \mathbin{\dot{\cup}} \ldots \mathbin{\dot{\cup}} C_{m_s}$ is an edge-disjoint union of two $1$-regular subgraphs of order $m_1+m_2+\ldots+m_s$. Therefore,   a  $(C_{m_1},C_{m_2}, \ldots, C_{m_t})$-decomposition of $G$ gives rise to a decomposition of $G$  into subgraphs consisting of a disjoint union of a   $(C_{m_{s+1}}, \ldots, C_{m_t})$-subgraph and two edge-disjoint $1$-regular subgraphs of order $m_1+m_2+\ldots+m_s$. Now, applying Lemma \ref{lem:Newtool-00}, we see $4G^{\bullet}$ admits an HOP $(C_2^{\langle \frac{m_1+m_2+\ldots+m_s}{2} \rangle}, C_{m_{s+1}},\ldots, C_{m_t})$-decomposition.
\end{proof}

The following lemma gives a sufficient condition for the multigraph $4G^{\bullet}$ to admit an HOP $(C_2^{\langle t \rangle})$-decomposition.

\begin{lem}{\label{G4C}}
Let $G$ be a simple graph. If $G$ admits a decomposition into $1$-regular subgraphs of order $2t$, then $4G^{\bullet}$ admits an HOP $(C_2^{\langle t \rangle})$-decomposition.
\end{lem}

\begin{proof}
Let $\F = \{H_1, H_2, \ldots, H_s\}$ be a decomposition of $G$ where each $H_i$ is a disjoint union of $t$ copies of $K_2$.
It is easy to see that $4K_2^{\bullet}$ decomposes into two HOP $(C_2)$-subgraphs, namely a directed black $2$-cycle, and a $2$-cycle with a pink and a blue edge. 
Therefore, for all $i \in \{1, \ldots, s\}$, the multigraph $4H_i^{\bullet}$ admits an HOP $(C_2^{\langle t \rangle})$-decomposition, and by Lemma~{\ref{Gtool3}}, the multigraph $4G^{\bullet}$ admits an HOP $(C_2^{\langle t \rangle})$-decomposition.
\end{proof}

The tools introduced below show how  2-regular HOP subgraphs with specific symmetry generate an HOP decomposition of $4K_n^{\bullet}$.

\begin{lem}{\label{lem:2Kn-d}}
Let $m = m_1 + \ldots + m_t$, where $2 = m_1 \leq m_2 \leq \ldots \leq m_t$ are integers. Assume $n$ is odd and  $V(2K_n^{\circ})=\{x_i: i\in \ZZ_{n-1} \}\cup \{x_{\infty}\}$.  Let $\rho_{\circ}$ be the permutation on $E(2K_n^{\circ})$ that preserves the color of the edges, and  is induced by  the permutation  $\rho=(x_{\infty})(x_0 \ x_1\  x_2 \ \ldots \ x_{n-2})$.
Let $F_1, F_2, \ldots, F_s$ be $(C_{m_1}, C_{m_2},\ldots, C_{m_t})$-subgraphs of $2K_n^{\circ}$ with the following properties. 
\begin{description}
\item [(A1)] Every cycle in $F_1, F_2, \ldots, F_s$ of length at least $3$ contains an even number of pink edges;
\item [(A2)] $F_1, F_2, \ldots, F_s$ jointly contain exactly two edges from each orbit of $\langle \rho_{\circ} \rangle$ 
of length $n-1$, namely,  a pair of edges of the form $(e,\rho_{\circ}^{\frac{n-1}{2}}(e))$; and
\item [(A3)] $F_1, F_2, \ldots, F_s$ jointly contain exactly one edge from each orbit of $\langle \rho_{\circ} \rangle$ 
of length $ \frac{n-1}{2}$.
\end{description}
Then $4K_n^{\bullet}$ admits an HOP $(C_{m_1}, C_{m_2},\ldots, C_{m_t})$-decomposition. 
\end{lem}
\begin{proof}
Observe that the group $\langle \rho_{\circ} \rangle$ has the following orbits on the edge set of $2K_n^{\circ}$:
\begin{itemize}
\item for each $d\in \{1,2,\ldots, \frac{ n-3}{2}\}$, we have a pink and a black orbit $\{x_ix_{i+d}: i\in \ZZ_{n-1} \}$;
\item a pink and a black orbit $\{x_ix_{i+{\frac{n-1}{2}}}: i=0,1,\ldots, \frac{n-3}{2} \}$; and
\item a pink and a black orbit  $\{x_ix_{\infty}: i\in \ZZ_{n-1} \}$. 
\end{itemize}
Note that the size of  each orbit is $n-1$, except for the pink and the black orbit consisting of edges of difference $\frac{n-1}{2}$; these two orbits have length $\frac{n-1}{2}$. 

Assume $O$ is an orbit of  $\langle \rho_{\circ} \rangle$. If the length of $O$ is $\frac{n-1}{2}$, then by {\bf (A3)} clearly $\{\rho_{\circ}^i(F_1), \ldots, \rho_{\circ}^i(F_s) : i=0,1,\ldots,\frac{n-3}{2}\}$ contains all edges of $O$. If  the length of $O$ is $n-1$, then for any $e\in O$, by {\bf (A2)}, we have that $e \in E(F_1\cup F_2\cup \ldots\cup F_s)$ implies $\rho_{\circ}^{\frac{n-1}{2}}(e)\in E(F_1\cup F_2\cup \ldots\cup F_s)$. Hence, $\{\rho_{\circ}^i(F_1),\ldots, \rho_{\circ}^i(F_s) : i=0,1,\ldots,\frac{n-3}{2}\}$ contains all edges of $O$.

Note that by {\bf (A2)} and {\bf (A3)} $|E(F_1\cup F_2\cup \ldots\cup F_s)|=2n$, and all edges in each orbit are covered  by $\{\rho_{\circ}^i(F_1),\ldots, \rho_{\circ}^i(F_s) : i=0,1,\ldots,\frac{n-3}{2}\}$. Hence, this set covers $n(n-1)$ edges, which is exactly the number of edges in $2K_n^{\circ}$, so no edge is covered twice. 
 Thus, $\D=\{\rho_{\circ}^i(F_1),\ldots, \rho_{\circ}^i(F_s) : i=0,1,\ldots,\frac{n-3}{2}\}$ forms a $(C_{m_1}, C_{m_2},\ldots, C_{m_t})$-decomposition  of $2K_n^{\circ}$. By Condition {\bf (A1)} and Lemma \ref{lem:Gtool2}, we see that $4K_n^{\bullet}$ admits an HOP $(C_{m_1}, C_{m_2},\ldots, C_{m_t})$-decomposition. 
\end{proof}

\begin{lem}{\label{lem:4Kn-d}}
Let $m = m_1 + \ldots + m_t$, where $2 = m_1 \leq m_2 \leq \ldots \leq m_t$ are integers. Assume $n$ is odd and $V(4K_n^{\bullet}) = \{x_i : i \in \mathbb{Z}_{n-1}\} \cup \{x_{\infty}\}$. Let $\rho_{\bullet}$ be the permutation on $E(4K_n^{\bullet})$ that preserves the color and orientation of the edges, and is induced by the permutation $\rho = (x_{\infty})(x_0 \ x_1 \ x_2 \ \ldots \ x_{n-2})$.
Let $F_1, F_2, \ldots, F_s$ be $(C_2, C_{m_2},\ldots, C_{m_t})$-subgraphs of $4K_n^{\bullet}$ with the following properties. 
\begin{description}
\item [(B1)] Every cycle in $F_1, F_2, \ldots, F_s$ satisfies Condition {\bf (C1)} of Definition {\rm{\ref{def}}};
\item [(B2)] $F_1, F_2, \ldots, F_s$ jointly contain exactly two edges of difference $\frac{n-1}{2}$, which appear in a $2$-cycle; and
\item [(B3)] $F_1, F_2, \ldots, F_s$ jointly contain exactly one edge from each of the orbits of $\langle \rho_{\bullet} \rangle$ corresponding to differences $ 1, 2 , \ldots,  \frac{ n-3}{2},  \infty$. 
\end{description}
Then there exists an HOP $(C_{2}, C_{m_2},\ldots, C_{m_t})$-decomposition of $4K_n^{\bullet}$. 
\end{lem}
\begin{proof}
Since $\rho_{\bullet}$ is a permutation that preserves the color and orientation of the edges, the group $\langle \rho_{\bullet} \rangle$ has the following orbits on the edge set of $4K_n^{\bullet}$:
\begin{itemize}
\item for each $d\in \{1,2,\ldots, \frac{ n-3}{2}\}$, we have a pink and a blue orbit $\{x_ix_{i+d}: i\in \ZZ_{n-1} \}$;
\item for each $d\in \{1,2,\ldots, n-2\}$, we have a black orbit $\{(x_i, x_{i+d}): i\in \ZZ_{n-1} \}$;
\item a pink and a blue orbit $\{ x_ix_{i+\frac{ n-1}{2}}: i=0,1,\ldots,\frac{n-3}{2}\}$;
\item a pink and a blue orbit  $\{x_ix_{\infty}: i\in \ZZ_{n-1} \}$; and
\item black orbits $\{(x_{\infty}, x_{i}): i\in \ZZ_{n-1} \}$ and $\{(x_i, x_{\infty}): i\in \ZZ_{n-1} \}$.
\end{itemize}
Note that the size of  each orbit is $n-1$, except for the pink and the blue orbit consisting of edges of difference $\frac{ n-1}{2}$; these two orbits have length $\frac{ n-1}{2}$. 

Let $e$ be an edge of $4K_n^{\bullet}$ of difference $\frac{n-1}{2}$. Then  by {\bf (B2)}, it appears only in a $2$-cycle, say  $C$. We may assume that $F_s$ is the $(C_{2}, C_{m_2},\ldots, C_{m_t})$-subgraph that contains $C$. Recolor the edges of  $C$ pink and blue. It is clear that $\{\rho_{\bullet}^i(F_s) : i=0,1,\ldots,\frac{n-3}{2}\}$ covers a pink and a blue copy of all edges of difference $\frac{n-1}{2}$. Now, form $F_s'$ from $F_s$ by recoloring the edges of $C$ black (as opposite arcs). Observe that $\{\rho_{\bullet}^{\frac{n-1}{2}+i}(F_s') : i=0,1,\ldots,\frac{n-3}{2}\}$ covers the black copies of all edges of difference $\frac{n-1}{2}$. Hence, $\{\rho_{\bullet}^{i}(F_s), \rho_{\bullet}^{\frac{n-1}{2}+i}(F_s') : i=0,1,\ldots,\frac{n-3}{2}\}$ covers all four copies of all edges of difference $\frac{n-1}{2}$. Moreover, since $F_1, F_2, \ldots, F_s$ jointly contain exactly one edge from each of the orbits, except for the orbits that correspond to difference $\frac{n-1}{2}$, we see that 
 $$\D=\big\{\rho_{\bullet}^i(F_s), \rho_{\bullet}^{\frac{n-1}{2}+i}(F_s'): i=0,1,\ldots,\textstyle{\frac{n-3}{2}} \big\}\cup\big\{\rho_{\bullet}^i(F_1), \rho_{\bullet}^i(F_2),\ldots, \rho_{\bullet}^i(F_{s-1}) : i\in \ZZ_{n-1}\big\}$$
 forms a $(C_{2}, C_{m_2},\ldots, C_{m_t})$-decomposition  for $4K_n^{\bullet}$. 
Since every cycle in $F_1, F_2, \ldots, F_s$ satisfies Condition {\bf (C1)} of Definition {\rm{\ref{def}}}, and $\rho_{\bullet}$ preserves the edge colors (and orientations),  decomposition $\D$ is  an HOP $(C_{2}, C_{m_2},\ldots, C_{m_t})$-decomposition of  $4K_n^{\bullet}$.
\end{proof}


The next lemma is proved similarly to Lemmas~\ref{lem:2Kn-d} and \ref{lem:4Kn-d}, so we leave the proof to the reader.
%
\begin{lem}{\label{lem:4Kn-e}}
Let $m = m_1 + \ldots + m_t$, where $2 = m_1 \leq m_2 \leq \ldots \leq m_t$ are integers. Assume $n$ is odd and $V(4K_n^{\bullet})=\{x_i: i\in \ZZ_{n-1} \}\cup \{x_{\infty}\}$.  Let $\rho_{\bullet}$ be the permutation on $E(4K_n^{\bullet})$ that preserves the color (and orientation) of the edges, and is induced by  the permutation  $\rho=(x_{\infty})(x_0 \ x_1\  x_2 \ \ldots \ x_{n-2})$. 
Let $F_1, F_2, \ldots, F_s$ be $(C_{m_1}, C_{m_2},\ldots, C_{m_t})$-subgraphs of $4K_n^{\bullet}$ with the following properties. 
\begin{description}
\item [(D1)] Every cycle in $F_1, F_2, \ldots, F_s$ satisfies Condition {\bf (C1)} of Definition {\rm{\ref{def}}};
\item[(D2)] $F_1, F_2, \ldots, F_s$ jointly contain exactly two edges from each orbit of $\langle \rho_{\bullet} \rangle$ 
of length $n-1$, namely, a pair of edges of the form $(e,\rho_{\bullet}^{\frac{n-1}{2}}(e))$; and
\item [(D3)] $F_1, F_2, \ldots, F_s$ jointly contain exactly one edge from each orbit of $\langle \rho_{\bullet} \rangle$ 
of length $ \frac{n-1}{2}$.
\end{description}
Then $\D=\{\rho_{\bullet}^i(F_1), \rho_{\bullet}^i(F_2),\ldots, \rho_{\bullet}^i(F_s): i=0,1,\ldots,\frac{n-3}{2}\}$ is an HOP $(C_{m_1}, C_{m_2},\ldots, C_{m_t})$-decomposition of $4K_n^{\bullet}$. 
\end{lem}

\section{The Generalized HOP with Two Round Tables}{\label{sec:5}}

To establish our results for the generalized HOP with two round tables, by Theorem~\ref{theo:Gtool1}, it suffices to find an HOP $(C_{m_1}, C_{m_2})$-decomposition of $4K_n^{\bullet}$. A significant portion of the proof is dedicated to the case where one of the cycles has length 2. In cases where both cycles have lengths of at least 3, we use the existing results in Theorem~\ref{theo:(Cm1,Cm2)-dec}. We first prove in Lemmas~\ref{lem:Newtool-1b} and~\ref{lem:Newtool-T2} that there exists an HOP $(C_2, C_m)$-decomposition of $4K_n^{\bullet}$ whenever $n \equiv 1 \pmod{2(m+2)}$ or $n \equiv m+2 \pmod{2(m+2)}$. These auxiliary results are then used to prove our first result, Theorem~\ref{theo-new-1}.

\subsection{HOP $(C_2, C_m)$-decompositions of $4K_n^{\bullet}$ when  $n \equiv 1 \pmod{2(m+2)}$}

\begin{lem}{\label{lem:Newtool-1b}}
Let $m\geq 3$, $k\geq1$ be integers, and let $n=2(m+2)k+1$. Then there exists an HOP $(C_2, C_m)$-decomposition of $4K_n^{\bullet}$. 
\end{lem}

\begin{proof} 
Let $V(K_n)=\{x_i: i\in \ZZ_n\}$, and let  $\rho=(x_0\ x_1\ \ldots \ x_{n-1})$ be a permutation on $V(K_n)$.
By the given labeling, each edge in $E(K_n)$ is of the form $x_ix_{i+d}$, where $d\in \{1,2,\ldots, (m+2)k\}$. Therefore, the group $\langle \rho \rangle$ acting on $E(K_n)$ has $(m+2)k$ orbits, each of size $n$.

Our goal is to construct $k$ subgraphs of $K_n$, namely $F_1, F_2, \ldots, F_k$, where each $F_i$ is a disjoint union of a $(C_m)$-subgraph and two edge-disjoint $1$-regular subgraphs of order $2$, and where $F_1, F_2, \ldots, F_k$ jointly contain exactly one edge of each difference in $\{1, 2, \ldots, (m+2)k\}$. Therefore, $\D = \{\rho^i(F_1), \rho^i(F_2), \ldots, \rho^i(F_k) : i \in \ZZ_n\}$ is a decomposition of $K_n$ that satisfies the condition in Lemma~\ref{lem:Newtool-00} (for $t = 1$ and $\alpha = 1$); hence, by Lemma~\ref{lem:Newtool-00}, the multigraph $4K_n^{\bullet}$ admits an HOP $(C_2, C_m)$-decomposition.

We start by constructing $m$-cycles. First, we build two paths, $P$ and $Q$, and then form the cycle $C$ by concatenating these paths appropriately. The construction depends on $m \bmod 4$, so we consider four cases: $m \equiv 0, 1, 2,$ or $3 \pmod{4}$.

\noindent {\bf{Case 1:}}  {\boldmath{$m \equiv 0 \pmod{4}$.}}
For $i\in\{1,2,\ldots,k\}$, define  walks $P_i$ and $Q_i$ as follows:
\begin{eqnarray*}
P_i &=& x_{1-i}\ x_i\ x_{-i}\ x_{2k+i}\ x_{-(2k+i)}\ x_{4k+i}\ x_{-(4k+i)} \ldots \\
&\ldots&   x_{\frac{m-12}{2}k+i}\ x_{-(\frac{m-12}{2}k+i)} \ x_{\frac{m-8}{2}k+i} \ x_{-(\frac{m-8}{2}k+i)}\ x_{\frac{m-4}{2}k+i}.
\\
Q_i &=& x_{-(\frac{m+4}{2}k+i)}\ \ x_{{\frac{m+4}{2}k+i}}\ \ x_{-(\frac{m+8}{2}k+i)}\ x_{\frac{m+8}{2}k+i}\ \ldots \\
&& \ldots  \ x_{-((m-2)k+i)}\ \ x_{(m-2)k+i} \ \ x_{-(mk+i)}\ \ x_{mk+i}. 
\end{eqnarray*}
Observe that $P_i$ and $Q_i$ are paths of length $\frac{m-2}{2}$.
 
Now, for $i\in\{1,2,\ldots,k\}$, we construct the following closed walk using paths $P_i$ and $Q_i$: 
$$C_i=P_i \ \textcolor{blue} {x_{\frac{m-4}{2}k+i}\ x_{-(\frac{m+4}{2}k+i)}}\  Q_i \ \textcolor{blue}{x_{mk+i} \ x_{1-i}}.$$
It can be seen that $C_i$ is a cycle of length $m$, and traverses edges of the following differences (in order): 
\begin{eqnarray*}
&2i-1,2i,2k+2i,4k+2i, 6k+2i, 8k+2i, \ldots,\\
&  (m-14)k+2i , (m-12)k+2i,(m-10)k+2i,(m-8)k+2i, (m-6)k+2i,  \\
& \textcolor{blue}{mk+2i}, mk-2i+1, (m-2)k-2i+1, (m-4)k-2i+1, \ldots, \\
&8k-2i+1, 6k-2i+1, 4k-2i+1, \textcolor{blue}{mk+2i-1}.
\end{eqnarray*}
Below, the differences in the sequences of  $C_1,C_2, \ldots, C_k$ are listed explicitly.

\begin{footnotesize}
\begin{eqnarray*}
C_1: &1, \ldots, (m-6)k+2, \textcolor{blue}{mk+2}, mk-1, (m-2)k-1,\ldots,4k-1, \textcolor{blue}{mk+1}\\
C_2: &3,\ldots,  (m-6)k+4, \textcolor{blue}{mk+4}, mk-3, (m-2)k-3,\ldots,4k-3, \textcolor{blue}{mk+3}\\
C_3: &5, \ldots, (m-6)k+6, \textcolor{blue}{mk+6}, mk-5, (m-2)k-5,\ldots,4k-5, \textcolor{blue}{mk+5}\\
&\vdots& \\
C_{k-2}: &2k-5,\ldots, (m-4)k-4, \textcolor{blue}{(m+2)k-4}, (m-2)k+5, (m-4)k+5,\ldots,2k+5, \textcolor{blue}{(m+2)k-5}\\
C_{k-1}: &2k-3, \ldots, (m-4)k-2, \textcolor{blue}{(m+2)k-2}, (m-2)k+3, (m-4)k+3,\ldots,2k+3, \textcolor{blue}{(m+2)k-3}\\
C_k: &2k-1,\ldots, (m-4)k, \textcolor{blue}{(m+2)k}, (m-2)k+1, (m-4)k+1,\ldots,2k+1, \textcolor{blue}{(m+2)k-1}
\end{eqnarray*}
\end{footnotesize}
Observe that $C_1, C_2, \ldots, C_k$ jointly contain exactly one edge of each difference in $\{1,2,3,\allowbreak \ldots ,(m+2)k\}$  except for
$(m-4)k+2j, \ {\rm { for }}\ j\in\{1,2,\ldots,2k\}.$
Note that we have $2k$ differences from $\{1,2,3,\ldots ,(m+2)k\}$ that are not covered by any $C_i$.  

First, assume \( m \geq 8 \). Define the following subgraphs of \( K_n \) (isomorphic to \( K_2 \)) that cover the leftover differences:
\begin{itemize}
  \item \( E_j^{(1)} = x_{-\left(\frac{m-6}{2}k + j\right)}\ x_{\frac{m-2}{2}k + j} \), for \( j \in \{1, 2, \ldots, k\} \);
  \item \( E_j^{(2)} = x_{-\left(\frac{m-4}{2}k + j\right)}\ x_{\frac{m}{2}k + j} \), for \( j \in \{1, 2, \ldots, k\} \).
\end{itemize}
Observe that none of the vertices in the intervals \( [x_{-(\frac{m-6}{2}k+1)},\, x_{-(\frac{m+4}{2}k)}] \) and \( [x_{\frac{m-2}{2}k+1},\, x_{\frac{m+4}{2}k}] \) are used in any \( C_i \), and that only vertices in these intervals are used to construct the subgraphs \( E_j^{(1)} \) and \( E_j^{(2)} \). Hence, each \( C_i \) is disjoint from \( E_j^{(1)} \) and \( E_j^{(2)} \), for all \( j \in \{1, 2, \ldots, k\} \).

For \( m = 4 \), the missing differences are \( 2j \), for \( j \in \{1, 2, \ldots, 2k\} \), which are covered by the following subgraphs:
\begin{itemize}
  \item \( E_j^{(1)} = x_{-(6k - j + 1)}\ x_{6k - j} \), for \( j \in \{1, 2, \ldots, k - 1\} \);
  \item \( E_k^{(1)} = x_{-k}\ x_{-3k} \); and
  \item \( E_j^{(2)} = x_{-(k + j)}\ x_{k + j} \), for \( j \in \{1,  2, \ldots, k\} \).
\end{itemize}


\noindent {\bf{Case 2:}} {\boldmath{$m\equiv 2\ ({\rm mod}\ 4)$.}}  For $i\in\{1,2,\ldots,k\}$, define  walks $P_i$ and $Q_i$ as follows:
\begin{eqnarray*}
P_i &=& x_{1-i}\ x_i\ x_{-i}\ x_{2k+i}\ x_{-(2k+i)}\ x_{4k+i}\ x_{-(4k+i)} \ldots \\
&\ldots&   x_{\frac{m-14}{2}k+i}\ x_{-(\frac{m-14}{2}k+i)} \ x_{\frac{m-10}{2}k+i}\ x_{-(\frac{m-10}{2}k+i)} \ x_{\frac{m-6}{2}k+i}\ x_{-(\frac{m-6}{2}k+i)}.
\\
Q_i &=& x_{\frac{m+2}{2}k+i}\ \ x_{-({\frac{m+6}{2}k+i})}\ \ x_{\frac{m+6}{2}k+i}\ \ldots \\
&& \ldots  \ x_{-((m-2)k+i)}\ \ x_{(m-2)k+i} \ \ x_{-(mk+i)}\ \ x_{mk+i}. 
\end{eqnarray*}
Observe that  $P_i$ and $Q_i$ are paths of length  $\frac{m-2}{2}$.
Now, for $i\in\{1,2,\ldots,k\}$, let: 
$$C_i=P_i \ \textcolor{blue}{x_{-(\frac{m-6}{2}k+i)}\  x_{\frac{m+2}{2}k+i}}\ Q_i \ \textcolor{blue}{x_{mk+i} \ x_{1-i}}.$$
It can be seen that $C_i$ is a cycle of length $m$ (see Figure~\ref{fig:starter1}), and traverses edges of the following differences (in order): 
\begin{eqnarray*}
&2i-1,2i,2k+2i,4k+2i, 6k+2i, 8k+2i, \ldots,\\
&  (m-14)k+2i , (m-12)k+2i,(m-10)k+2i,(m-8)k+2i, (m-6)k+2i,  \\
& \textcolor{blue}{(m-2)k+2i}, mk-2i+1, (m-2)k-2i+1, (m-4)k-2i+1, \ldots, \\
&8k-2i+1, 6k-2i+1, 4k-2i+1, \textcolor{blue}{mk+2i-1}
\end{eqnarray*}
The differences in \( C_1, C_2, \ldots, C_k \) are listed below.
\begin{footnotesize}
\begin{eqnarray*}
C_1: &1,2, \ldots, (m-6)k+2, \textcolor{blue}{(m-2)k+2}, mk-1, (m-2)k-1,\ldots,4k-1, \textcolor{blue}{mk+1}\\
C_2: &3,4,\ldots,  (m-6)k+4, \textcolor{blue}{(m-2)k+4}, mk-3, (m-2)k-3,\ldots,4k-3, \textcolor{blue}{mk+3}\\
C_3: &5,6, \ldots, (m-6)k+6, \textcolor{blue}{(m-2)k+6}, mk-5, (m-2)k-5,\ldots,4k-5, \textcolor{blue}{mk+5}\\
&\vdots& \\
C_{k-2}: &2k-5,2k-4,\ldots, (m-4)k-4, \textcolor{blue}{mk-4}, (m-2)k+5, (m-4)k+5,\ldots,2k+5, \textcolor{blue}{(m+2)k-5}\\
C_{k-1}: &2k-3,2k-2, \ldots, (m-4)k-2, \textcolor{blue}{mk-2}, (m-2)k+3, (m-4)k+3,\ldots,2k+3, \textcolor{blue}{(m+2)k-3}\\
C_k: &2k-1,2k,\ldots, (m-4)k, \textcolor{blue}{mk}, (m-2)k+1, (m-4)k+1,\ldots,2k+1, \textcolor{blue}{(m+2)k-1}
\end{eqnarray*}
\end{footnotesize}
Observe that $C_1, C_2, \ldots, C_k$ jointly contain exactly one edge of each difference in $\{1,2,3,\ldots ,\allowbreak (m+2)k\}$  except for the following differences:
\begin{enumerate}[(i)]
\item $(m-4)k+2j$, for $j\in\{1,2,\ldots,k\}$; and
\item $mk+2j$, for $j\in\{1,2,\ldots,k\}$.
\end{enumerate}
The above $2k$ leftover differences from the set $\{1, 2, 3, \ldots, (m+2)k\}$ are covered by the following subgraphs of $K_n$ (each isomorphic to $K_2$):
\begin{itemize}
  \item \( E_j^{(1)} = x_{-\left(\frac{m - 4}{2}k + j\right)}\ x_{\frac{m - 4}{2}k + j} \), for \( j \in \{1, 2, \ldots, k\} \);
  \item \( E_j^{(2)} = x_{-\left(\frac{m}{2}k + j\right)}\ x_{\frac{m}{2}k + j} \), for \( j \in \{1, 2, \ldots, k\} \).
\end{itemize}
Observe that none of the vertices in the intervals \( [x_{-(\frac{m-4}{2}k + 1)},\, x_{-(\frac{m+6}{2}k)}] \) and \( [x_{\frac{m-4}{2}k + 1},\, x_{\frac{m+2}{2}k}] \) are used in any \( C_i \), and that only vertices from these intervals are used to construct the subgraphs \( E_j^{(1)} \) and \( E_j^{(2)} \). Hence, each \( C_i \) is disjoint from all subgraphs \( E_j^{(1)} \) and \( E_j^{(2)} \), for \( j \in \{1, 2, \ldots, k\} \).

\begin{figure}[hbt!]
 \centering
    \includegraphics[width=0.75\linewidth]{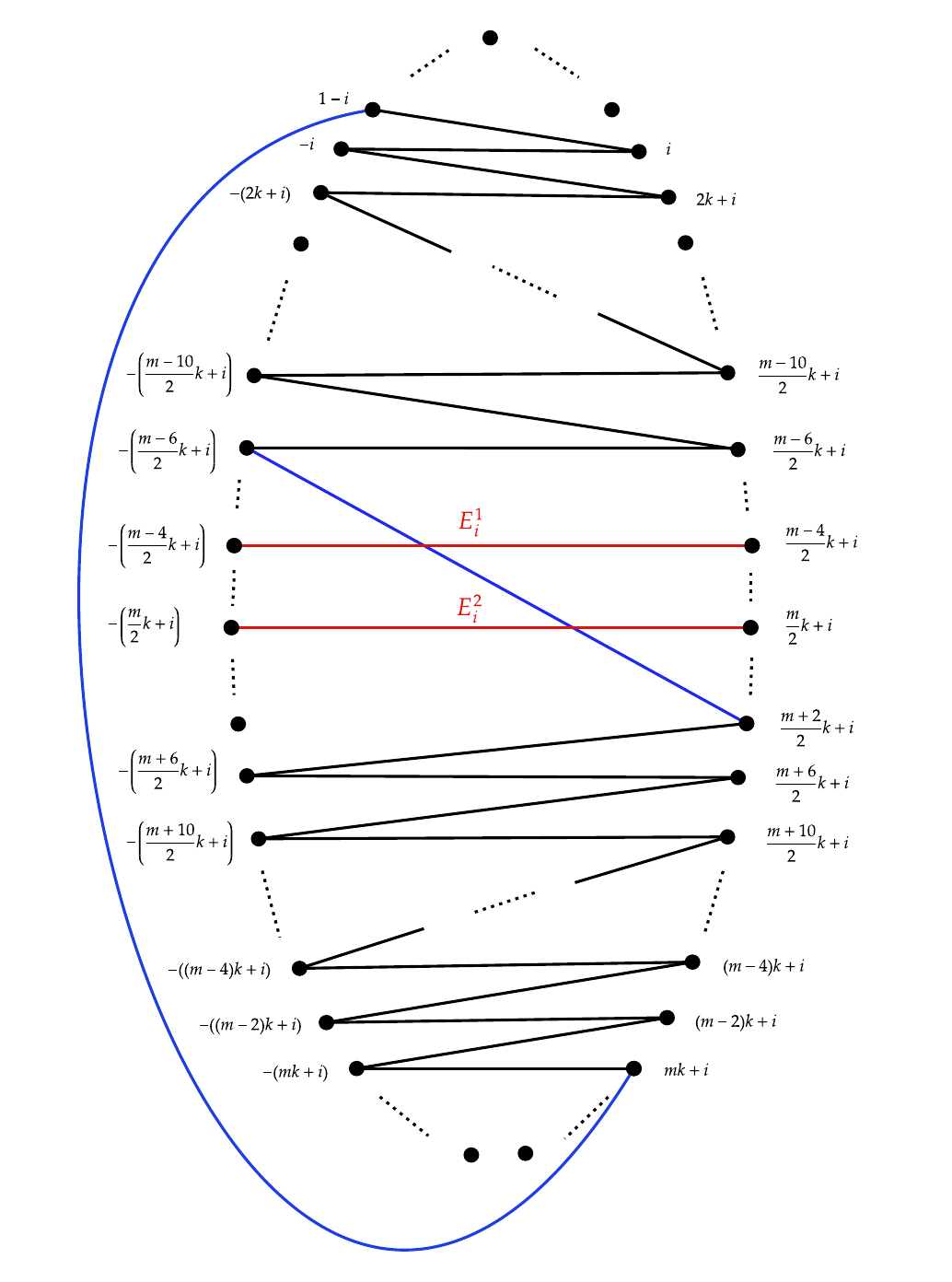}
    \caption{The $F_i$ subgraph in Lemma~\ref{lem:Newtool-1b} for $m\equiv 2\ ({\rm mod}\ 4)$.}
    \label{fig:starter1}
  \end{figure}

\noindent {\bf{Case 3:}}  {\boldmath{$m\equiv 1\ ({\rm mod}\ 4)$.}}  For $i\in\{1,2,\ldots,k\}$, define a walk $P_i$ as follows:
\begin{eqnarray*}
P_i &=& x_{1-i}\ x_i\ x_{-i}\ x_{2k+i}\ x_{-(2k+i)}\ x_{4k+i}\ x_{-(4k+i)} \ldots \\
&\ldots&   x_{\frac{m-9}{2}k+i}\ x_{-(\frac{m-9}{2}k+i)} \ x_{\frac{m-5}{2}k+i}\ x_{-(\frac{m-5}{2}k+i)} \ x_{\frac{m-1}{2}k+i}.
\end{eqnarray*}
We see that  $P_i$ is a path of length $\frac{m+1}{2}$. 

For \( m = 5 \), let \( Q_i = x_{4k + 3i - 1} \). Next, assume \( m \geq 9 \), and for \( i \in \{1, 2, \ldots, k\} \), define a walk \( Q_i \) as follows:
{\small{\begin{eqnarray*}
Q_i &=& x_{-({\frac{m+11}{2}k+i})}\ \ x_{\frac{m+3}{2}k+i}\ \ x_{-(\frac{m+15}{2}k+i)}\ \ x_{\frac{m+7}{2}k+i}\ \ldots \\
&& \ldots \ x_{-((m-1)k+i)}\ \ x_{(m-5)k+i}\ \ x_{-((m+1)k+i)} \ \ x_{(m-3)k+i}\ \ x_{(m-1)k+3i-1}. 
\end{eqnarray*}}}
Observe $Q_i$ is a path of length  $\frac{m-5}{2}$. 
Now, for $i\in\{1,2,\ldots,k\}$, let:
$$C_i=P_i \ \textcolor{blue}{x_{\frac{m-1}{2}k+i}\  x_{-(\frac{m+11}{2}k+i)}}\ Q_i \ \textcolor{blue}{x_{(m-1)k+3i-1} \ x_{1-i}}.$$
It can be seen that $C_i$ is a cycle of length $m$, and traverses edges of the following differences (in order): 
\begin{eqnarray*}
&2i-1,2i,2k+2i,4k+2i, 6k+2i, 8k+2i, \ldots,\\
&  (m-9)k+2i,(m-7)k+2i,(m-5)k+2i, (m-3)k+2i,  \\
& \textcolor{blue}{(m-1)k-2i+1}, (m-3)k-2i+1, (m-5)k-2i+1, (m-7)k-2i+1, \ldots, \\
&10k-2i+1, 8k-2i+1, 6k-2i+1, 2k+2i-1, \textcolor{blue}{d}.
\end{eqnarray*}
Note that for $i\in\{1,2,\ldots,\floor{\frac{3k+2}{4}}\}$, the difference \textcolor{blue}{$d=(m-1)k+4i-2$}, and it is even; for $i\in\{\floor{\frac{3k+2}{4}}+1,\ldots,k\}$, the difference \textcolor{blue}{$d=(m+1)k+4(k-i)+3$}, and it is odd. 
Moreover, for all $i\in\{1,2,\ldots,k\}$,  the difference \textcolor{blue}{$(m-1)k-2i+1$} is odd.   
The explicit differences in  $C_1,C_2, \ldots, C_k$ are listed below.

{\footnotesize{
\begin{eqnarray*}
C_1: &1,2, \ldots, (m-3)k+2, \textcolor{blue}{(m-1)k-1}, (m-3)k-1, \ldots,2k+1, \textcolor{blue}{(m-1)k+2}\\
C_2: &3,4,\ldots,  (m-3)k+4, \textcolor{blue}{(m-1)k-3}, (m-3)k-3, \ldots,2k+3, \textcolor{blue}{(m-1)k+6}\\
C_3: &5,6, \ldots, (m-3)k+6, \textcolor{blue}{(m-1)k-5}, (m-3)k-5, \ldots,2k+5, \textcolor{blue}{(m-1)k+10}\\
&\vdots& \\
C_{k-2}: &2k-5,2k-4,\ldots, (m-1)k-4, \textcolor{blue}{(m-3)k+5}, (m-5)k+5, \ldots,4k-5, \textcolor{blue}{(m+1)k+11}\\
C_{k-1}: &2k-3,2k-2, \ldots, (m-1)k-2, \textcolor{blue}{(m-3)k+3}, (m-5)k+3, \ldots,4k-3, \textcolor{blue}{(m+1)k+7}\\
C_k: &2k-1,2k,\ldots, (m-1)k, \textcolor{blue}{(m-3)k+1}, (m-5)k+1, \ldots,4k-1, \textcolor{blue}{(m+1)k+3}
\end{eqnarray*}}}
Observe that $C_1, C_2, \ldots, C_k$ jointly contain exactly one edge of each difference in $\{1,2,3,\ldots , \allowbreak (m+2)k\}$  except for the following differences:
\begin{enumerate}[(i)]
\item $(m-1)k+2j-1$, for $j\in\{1,2,\ldots,k\}$;
\item $(m-1)k+4j$, for $j\in\{1,2,\ldots,\floor{\frac{3k}{4}}\}$; and
\item $(m+1)k+4(k-i)+1$, for $j\in\{\floor{\frac{3k}{4}}+1,\ldots,k\}$.
\end{enumerate}
The above differences are covered by the following subgraphs:
\begin{itemize}
\item \( E_j^{(1)} = x_{-\left(\frac{m+1}{2}k + j - 1\right)}\ x_{\left(\frac{m-3}{2}k + j\right)} \), for \( j \in \{1, 2, \ldots, k\} \);
\item \( E_j^{(2)} = x_{-\left(\frac{m-3}{2}k + 3j\right)}\ x_{\left(\frac{m+1}{2}k + j\right)} \), for \( j \in \{1, 2, \ldots, k\} \).
\end{itemize}
Note that the edge in $E_j^{(1)}$ is of difference  $(m-1)k+2j-1$. The edge in $E_j^{(2)}$, for $j\in\{1,2,\ldots,\floor{\frac{3k}{4}}\}$, is of difference  $(m-1)k+4j$, and for $j\in\{\floor{\frac{3k}{4}}+1,\ldots,k\}$, it is of difference $(m+1)k+4(k-i)+1$.

Observe that none of the vertices in the intervals 
\([x_{-(\frac{m-3}{2}k+1)}, x_{-(\frac{m+11}{2}k)}]\), 
\([x_{\frac{m-3}{2}k+1}, x_{\frac{m-1}{2}k}]\), and 
\([x_{\frac{m+1}{2}k+1}, x_{\frac{m+3}{2}k}]\) 
are used in any \(C_i\), and only vertices from these intervals are used to construct the subgraphs \(E_j^{(1)}\) and \(E_j^{(2)}\). 
Hence, each \(C_i\) is disjoint from all subgraphs \(E_j^{(1)}\) and \(E_j^{(2)}\), for \(j \in \{1,2,\ldots,k\}\).

\noindent {\bf{Case 4:}} {\boldmath{$m\equiv 3\ ({\rm mod}\ 4)$.}}
For $i\in\{1,2,\ldots,k\}$, define a walk $P_i$ as follows:
\begin{eqnarray*}
P_i &=& x_{1-i}\ x_i\ x_{-i}\ x_{2k+i}\ x_{-(2k+i)}\ x_{4k+i}\ x_{-(4k+i)} \ldots \\
&\ldots&   x_{\frac{m-11}{2}k+i}\ x_{-(\frac{m-11}{2}k+i)} \ x_{\frac{m-7}{2}k+i}\ x_{-(\frac{m-7}{2}k+i)} \ x_{\frac{m-3}{2}k+i}.
\end{eqnarray*}
Observe that $P_i$ is a path of length $\frac{m-1}{2}$. 

For $m=3$, let $Q_i=x_{2k+3i-1}$. Now, assume $m\geq 7$, and for $i\in\{1,2,\ldots,k\}$, define a walk $Q_i$ as follows:

{\small{\begin{eqnarray*}
Q_i &=& x_{-({\frac{m+9}{2}k+i})}\ \ x_{\frac{m+1}{2}k+i}\ \ x_{-(\frac{m+13}{2}k+i)}\ \ x_{\frac{m+5}{2}k+i}\ \ldots \\
&& \ldots \ x_{-((m-1)k+i)}\ \ x_{(m-5)k+i}\ \ x_{-((m+1)k+i)} \ \ x_{(m-3)k+i}\ \ x_{(m-1)k+3i-1}. 
\end{eqnarray*}}}
Observe $Q_i$ is a path of length $\frac{m-3}{2}$.
Now, for $i\in\{1,2,\ldots,k\}$, let: 
$$C_i=P_i \ \textcolor{blue}{x_{\frac{m-3}{2}k+i}\  x_{-(\frac{m+9}{2}k+i)}}\ Q_i \ \textcolor{blue}{x_{(m-1)k+3i-1} \ x_{1-i}}.$$
Observe $C_i$ is a cycle of length $m$, and traverses edges of the following differences (in order): 
\begin{eqnarray*}
&2i-1,2i,2k+2i,4k+2i, 6k+2i, 8k+2i, \ldots,\\
&  (m-11)k+2i , (m-9)k+2i,(m-7)k+2i,(m-5)k+2i,  \\
& \textcolor{blue}{(m+1)k-2i+1}, (m-1)k-2i+1, (m-3)k-2i+1, (m-5)k-2i+1,  \ldots, \\
&10k-2i+1, 8k-2i+1, 6k-2i+1, 2k+2i-1, \textcolor{blue}{d}
\end{eqnarray*}

Note that for $i\in\{1,2,\ldots,\floor{\frac{3k+2}{4}}\}$, the difference \textcolor{blue}{$d=(m-1)k+4i-2$}, and it is even; for $i\in\{\floor{\frac{3k+2}{4}}+1,\ldots,k\}$, the difference \textcolor{blue}{$d=(m+1)k+4(k-i)+3$}, and it is odd. 
Moreover, for all $i\in\{1,2,\ldots,k\}$,  the difference \textcolor{blue}{$(m+1)k-2i+1$} is odd.   
Below, the differences in \( C_1, C_2, \ldots, C_k \) are listed explicitly.

{\footnotesize{
\begin{eqnarray*}
C_1: &1,2, \ldots, (m-5)k+2, \textcolor{blue}{(m+1)k-1}, (m-1)k-1, \ldots,6k-1, 2k+1, \textcolor{blue}{(m-1)k+2}\\
C_2: &3,4,\ldots,  (m-5)k+4, \textcolor{blue}{(m+1)k-3}, (m-1)k-3, \ldots,6k-3, 2k+3, \textcolor{blue}{(m-1)k+6}\\
C_3: &5,6, \ldots, (m-5)k+6, \textcolor{blue}{(m+1)k-5}, (m-1)k-5, \ldots,6k-5, 2k+5, \textcolor{blue}{(m-1)k+10}\\
&\vdots& \\ 
C_{k-2}: &2k-5,2k-4,\ldots, (m-3)k-4, \textcolor{blue}{(m-1)k+5}, (m-3)k+5, \ldots,4k+5, 4k-5, \textcolor{blue}{(m+1)k+11}\\
C_{k-1}: &2k-3,2k-2, \ldots, (m-3)k-2, \textcolor{blue}{(m-1)k+3}, (m-3)k+3, \ldots,4k+3, 4k-3, \textcolor{blue}{(m+1)k+7}\\
C_k: &2k-1,2k,\ldots, (m-3)k, \textcolor{blue}{(m-1)k+1}, (m-3)k+1, \ldots,4k+1, 4k-1, \textcolor{blue}{(m+1)k+3}
\end{eqnarray*}}}

Observe that $C_1, C_2, \ldots, C_k$ jointly contain exactly one edge of each difference in $\{1,2,3,\allowbreak \ldots ,(m+2)k\}$  except for the following differences:
\begin{enumerate}[(i)]
\item $(m-3)k+2j$, for $j\in\{1,2,\ldots,k\}$;
\item $(m-1)k+4j$, for $j\in\{1,2,\ldots,\floor{\frac{3k}{4}}\}$;
\item $(m+1)k+4(k-i)+1$, for $j\in\{\floor{\frac{3k}{4}}+1,\ldots,k\}$.
\end{enumerate}
There are  $2k$ differences from the set $\{1,2,3,\ldots ,(m+2)k\}$ that are not covered by any $C_i$. Next, we define subgraphs that cover the leftover differences. 
First assume $m\geq 7$, and for all $j\in\{1,2,\ldots,k\}$, let $E_j^{(1)}$ and $E_j^{(2)}$ denote the following subgraphs of $K_n$ (isomorphic to $K_2$):
\begin{itemize}
\item $E_j^{(1)}= x_{-(\frac{m-5}{2}k+j)}\ \ x_{\frac{m-1}{2}k+j}$;  and
\item $E_j^{(2)}= x_{-(\frac{m+3}{2}k+3j)}\ \ x_{\frac{m-5}{2}k+j}$.
\end{itemize}
Note that the edge in \(E_j^{(1)}\) has difference \((m - 3)k + 2j\). The edge in \(E_j^{(2)}\), for \(j \in \{1, 2, \ldots, \lfloor \frac{3k}{4} \rfloor\}\), has difference \((m - 1)k + 4j\), and for \(j \in \{\lfloor \frac{3k}{4} \rfloor + 1, \ldots, k\}\), it has difference \((m + 1)k + 4(k - j) + 1\).

Observe that none of the vertices in the intervals 
\([x_{-(\frac{m-5}{2}k+1)}, x_{-(\frac{m+9}{2}k)}]\), 
\([x_{\frac{m-1}{2}k+1}, x_{\frac{m+1}{2}k}]\), and 
\([x_{\frac{m-5}{2}k+1}, x_{\frac{m-3}{2}k}]\) 
are used in any \(C_i\), and only vertices from these intervals are used to construct the graphs \(E_j^{(1)}\) and \(E_j^{(2)}\). 
Hence, each \(C_i\) is disjoint from all subgraphs \(E_j^{(1)}\) and \(E_j^{(2)}\), for \(j \in \{1, 2, \ldots, k\}\).

For $m=3$, the missing differences are as follows:
\begin{enumerate}[(i)]
\item $2j$, for $j\in\{1,2,\ldots,k\}$;
\item $2k+4j$, for $j\in\{1,2,\ldots,\floor{\frac{3k}{4}}\}$;
\item $8k-4j+1$, for $j\in\{\floor{\frac{3k}{4}}+1,\ldots,k\}$.
\end{enumerate}
The following subgraphs cover the differences listed above: 
\begin{itemize}
\item \( E_j^{(1)} = x_{-(2k + j)}\, x_{-(2k - j)} \), with difference \( 2j \), for \( j \in \{1, 2, \ldots, k\} \);

\item \( E_j^{(2)} = x_{-(k + 3j)}\, x_{k + j} \), which has difference \( 2k + 4j \) for \( j \in \left\{1, 2, \ldots, \left\lfloor \frac{3k}{4} \right\rfloor \right\} \), and difference \( 8k - 4j + 1 \) for \( j \in \left\{ \left\lfloor \frac{3k}{4} \right\rfloor + 1, \ldots, k \right\} \).

\end{itemize}
%
%
In all four cases above, we have shown how to construct \( m \)-cycles \( C_1, C_2, \ldots, C_k \), and subgraphs  \( E_1^{(1)}, E_2^{(1)}, \ldots, E_k^{(1)} \) and \( E_1^{(2)}, E_2^{(2)}, \ldots, E_k^{(2)} \);  now, for \( i \in \{1, 2, \ldots, k\} \), let \( F_i = C_i \mathbin{\dot{\cup}} (E_i^{(1)} \cup E_i^{(2)}) \). 
Since \( F_1, F_2, \ldots, F_k \) jointly contain exactly one edge of each difference in \( \{1, 2, \ldots, (m + 2)k\} \), we see that
\(
\mathcal{D} = \{ \rho^i(F_1), \rho^i(F_2), \ldots, \rho^i(F_k) : i \in \mathbb{Z}_n \}
\)
is a decomposition of \( K_n \) that satisfies the conditions of Lemma~\ref{lem:Newtool-00}. Applying Lemma~\ref{lem:Newtool-00}, we conclude that the multigraph \( 4K_n^{\bullet} \) admits an HOP \((C_2, C_m)\)-decomposition. 
\end{proof}


\subsection{HOP $(C_2, C_m)$-Decompositions of $4K_n^{\bullet}$ when $ n \equiv m + 2\ (\mathrm{mod}\ 2(m + 2))$ }

\begin{lem}{\label{lem:Newtool-T2}}
Let $m \geq 3$ be an odd integer, and let $k \geq 0$ be an integer. Let $n = (2k + 1)(m + 2)$. Then there exists an HOP $(C_2, C_m)$-decomposition of $4K_n^{\bullet}$.
\end{lem}
\begin{proof} 
Notice that \( K_{(2k+1)(m+2)} = H_1 \oplus H_2 \), where \( H_1 \) is the disjoint union of \( 2k + 1 \) copies of \( K_{m+2} \), and \( H_2 \) is the complete equipartite graph with \( 2k + 1 \) parts of size \( m + 2 \), that is, \( K_{(2k+1)[m+2]} \). By Lemma~\ref{lems-1starter}, the multigraph \( 4K_{m+2}^{\bullet} \) admits an HOP \((C_2, C_m)\)-factorization. Consequently, by Lemma~\ref{Gtool3}, \( 4H_1^{\bullet} \) admits an HOP \((C_2, C_m)\)-decomposition. Hence, by Lemma~\ref{Gtool3}, it suffices to construct an HOP \((C_2, C_m)\)-decomposition of \( 4K_{(2k+1)[m+2]}^{\bullet} \).


For simplicity, let $\ell = 2k + 1$.  The complete equipartite graph $K_{\ell[m+2]}$ is a circulant graph, Circ$(n; \pm S)$, with vertex set $\{x_i: i\in \ZZ_n\}$, where $n=(m+2)\ell$, and edge set $\{x_ix_{i+d}: i\in \ZZ_n, d\in S\}$, where 
$$S=\Big\{1,2,\ldots, \ell -1, \ell+1, \ldots,2\ell-1, 2\ell+1, \ldots, \frac{(m+2)\ell-1}{2}\Big\}.$$
Let  $\rho=(x_0\ x_1\ \ldots \ x_{n-1})$ be a permutation on $V(K_{\ell[m+2]})$.
Observe that \( S \) contains all differences from 1 to \( \frac{(m+2)\ell - 1}{2} \), except those in the set \( \{\ell, 2\ell, \ldots, \tfrac{(m+1)}{2}\ell\} \). Therefore, the group \( \langle \rho \rangle \) acting on \( E(K_{\ell[m+2]}) \) has $(m+2)k$ orbits, each of size \( n = (m+2)\ell \).



We aim to construct $k$ subgraphs of $K_{\ell[m+2]}$, namely $F_1,F_2, \ldots, F_k$, where each $F_i$ is a disjoint union of a $(C_m)$-subgraph and two edge-disjoint $1$-regular subgraphs of order $2$, and $F_1,F_2, \ldots, F_k$ jointly contain exactly one edge of each difference in $S$.  Therefore, $\D=\{\rho^i(F_1), \allowbreak \rho^i(F_2), \ldots, \rho^i(F_k): i\in \ZZ_n\}$ is a decomposition of $K_{\ell[m+2]}$ that satisfies the condition of  Lemma \ref{lem:Newtool-00} (for $\alpha=1$ and $t=1$); hence, by Lemma \ref{lem:Newtool-00}, the multigraph $4K_{\ell[m+2]}^{\bullet}$ admits an HOP $(C_2, C_m)$-decomposition.

We construct a $(C_m)$-subgraph by first building two paths, $P$ and $Q$, and then joining them to form the cycle $C_m$. Since \( m \) is odd, it suffices to  consider the following two cases: \( m \equiv 1 \pmod{4} \) and \( m \equiv 3 \pmod{4} \).

\noindent {\bf{Case 1:}}  {\boldmath{$m\equiv 1\ ({\rm mod}\ 4)$.}}
 For $i\in\{1,2,\ldots,k\}$, define a walk $P_i$ as follows:
\begin{eqnarray*}
P_i &=& x_{1-i}\ x_i\ x_{-i}\ x_{(\ell-1)+i}\ x_{-((\ell+1)+i)}\ x_{2(\ell-1)+i}\ x_{-(2(\ell+1)+i)} \ldots \\
&\ldots&   x_{\frac{m-9}{4}(\ell-1)+i}\ x_{-(\frac{m-9}{4}(\ell+1)+i)} \ x_{\frac{m-5}{4}(\ell-1)+i}\ x_{-(\frac{m-5}{4}(\ell+1)+i)} \ x_{\frac{m-1}{4}(\ell-1)+i}.
\end{eqnarray*}
Observe that $P_i$ is a path of length $\frac{m+1}{2}$.

{\bf{Subcase 1.1:}}
Assume $m \geq 9$, and for  $i \in \{1, 2, \ldots, k\}$, define a walk $Q_i$ as follows:

{\scriptsize{\begin{eqnarray*}
Q_i &=& x_{-({\frac{m+7}{4}(\ell-1)+\frac{m-1}{2}+3i})}\ \ x_{\frac{m+7}{4}(\ell+1)-\frac{m-1}{2}-i}\ \ x_{-(\frac{m+11}{4}(\ell-1)+\frac{m-1}{2}+3i)}\ \ x_{\frac{m+11}{4}(\ell+1)-\frac{m-1}{2}-i}\  \ldots \\
&& \ldots \ x_{-({\frac{m-3}{2}(\ell-1)+\frac{m-1}{2}+3i})}\ \ x_{{\frac{m-3}{2}(\ell+1)-\frac{m-1}{2}-i}} \ \ x_{-({\frac{m-1}{2}(\ell-1)+\frac{m-1}{2}+3i})}\ \ x_{{\frac{m-1}{2}(\ell+1)-\frac{m-1}{2}-i}}\ \ x_{{\frac{m+1}{2}\ell+i}}. 
\end{eqnarray*}}}
Observe that $Q_i$ is a path of length  $\frac{m-5}{2}$.
Assume that $m \geq 9$ and  for $i\in\{1,2,\ldots,k\}$ let: 
$$C_i=P_i \ \textcolor{blue}{x_{\frac{m-1}{4}(\ell-1)+i}\  x_{-({\frac{m+7}{4}(\ell-1)+\frac{m-1}{2}+3i})}}\ Q_i \ \textcolor{blue}{x_{{\frac{m+1}{2}\ell+i}} \ x_{1-i}}.$$
It can be seen that $C_i$ is a cycle of length $m$ (see Figure~\ref{fig:starter2}), and traverses edges of the following differences: 
\begin{eqnarray*}
&2i-1,2i, \ell+2i-1, 2\ell+2i,3\ell+2i-1, 4\ell+2i, 5\ell+2i-1, \ldots,\\
&  \frac{m-9}{2}\ell+2i, \frac{m-7}{2}\ell+2i-1, \frac{m-5}{2}\ell+2i, \frac{m-3}{2}\ell+2i-1,  \\
& \textcolor{blue}{\frac{m+1}{2}\ell-4i+2}, \frac{m-3}{2}\ell-2i,  \frac{m-5}{2}\ell-2i+1,  \frac{m-7}{2}\ell-2i, \ldots, \\
&5\ell-2i, 4\ell-2i+1, 3\ell-2i, \ell+2i, \textcolor{blue}{d}
\end{eqnarray*}
Note that for $i\in\{1,2,\ldots,\floor{\frac{k+1}{2}}\}$, the difference \textcolor{blue}{$d= \frac{m+1}{2}\ell+2i-1$}, and it is even; for $i\in\{\floor{\frac{k+1}{2}}+1,\ldots,k\}$, the difference \textcolor{blue}{$d= \frac{m+3}{2}\ell-2i+1$}, and it is odd. 
Moreover, for all $i\in\{1,2,\ldots,k\}$,  the difference \textcolor{blue}{$\frac{m+1}{2}\ell-4i+2$} is odd.   
The explicit differences in $C_1,C_2, \ldots, C_k$ are listed below.

{\footnotesize{
\begin{eqnarray*}
C_1: &1,2, \ell+1,\ldots,  \frac{m-3}{2}\ell+1,\textcolor{blue}{\frac{m+1}{2}\ell-2}, \frac{m-3}{2}\ell-2, \ldots,3\ell-2, \ell+2, \textcolor{blue}{\frac{m+1}{2}\ell+1}\\
C_2: &3,4, \ell+3,\ldots,   \frac{m-3}{2}\ell+3, \textcolor{blue}{\frac{m+1}{2}\ell-6},  \frac{m-3}{2}\ell-4, \ldots,3\ell-4, \ell+4, \textcolor{blue}{\frac{m+1}{2}\ell+3}\\
C_3: &5,6, \ell+5,\ldots, \frac{m-3}{2}\ell+5, \textcolor{blue}{\frac{m+1}{2}\ell-10}, \frac{m-3}{2}\ell-6, \ldots,3\ell-6, \ell+6, \textcolor{blue}{\frac{m+1}{2}\ell+5}\\
&\vdots& \\
C_{k-2}: &2k-5,2k-4,2\ell-6,\ldots,  \frac{m-1}{2}\ell-6, \textcolor{blue}{\frac{m-3}{2}\ell+12},  \frac{m-5}{2}\ell+5, \ldots,2\ell+5, 2\ell-5, \textcolor{blue}{\frac{m+1}{2}\ell+6}\\
C_{k-1}: &2k-3,2k-2, 2\ell-4, \ldots,\frac{m-1}{2}\ell-4, \textcolor{blue}{\frac{m-3}{2}\ell+8}, \frac{m-5}{2}\ell+3, \ldots,2\ell+3, 2\ell-3, \textcolor{blue}{\frac{m+1}{2}\ell+4}\\
C_k: &2k-1,2k, 2\ell-2,\ldots, \frac{m-1}{2}\ell-2, \textcolor{blue}{\frac{m-3}{2}\ell+4},  \frac{m-5}{2}\ell+1, \ldots,2\ell+1, 2\ell-1, \textcolor{blue}{\frac{m+1}{2}\ell+2}
\end{eqnarray*}}}

Observe that $C_1, C_2, \ldots, C_k$ jointly contain exactly one edge of each difference in $S$  except for the following differences:
\begin{enumerate}[(i)]
\item $\frac{m-1}{2}\ell+2j$, for $j\in\{1,2,\ldots,k\}$;
\item $\frac{m-3}{2}\ell+4j-2$, for $j\in\{1,2,\ldots,k\}$.
\end{enumerate}
Note that we have $2k$ differences from $S$ that are not covered by any $C_i$. 
For all $j\in\{1,2,\ldots,k\}$, let $E_j^{(1)}$ and $E_j^{(2)}$ denote the follwoing subgraphs that cover the leftover differences:
\begin{itemize}
\item $E_j^{(1)} = x_{-(\frac{m-3}{4}(\ell+1)+j)}\ \ x_{\frac{m+1}{4}(\ell-1)+j+1}$; and
\item $E_j^{(2)}= x_{-(\frac{m-3}{4}(\ell+1)+3j-2)}\ \ x_{\frac{m-3}{4}(\ell-1)+j}$.
\end{itemize}
Note that for $j\in\{1,2,\ldots,k\}$, the edge in $E_j^{(1)}$ is of difference $\frac{m-1}{2}\ell+2j$, and the edge in $E_j^{(2)}$ is of difference  $\frac{m-3}{2}\ell+4j-2$.

Moreover, observe that none of the vertices in the intervals $[x_{-(\frac{m - 3}{4}(\ell + 1))},\ x_{-(\frac{m + 7}{4}(\ell + 1) - 2)}]$, $[x_{\frac{m + 1}{4}(\ell - 1) + 1},\ \allowbreak x_{\frac{m + 5}{4}(\ell - 1) + 3}]$, and $[x_{\frac{m - 3}{4}(\ell - 1) + 1},\ x_{\frac{m - 1}{4}(\ell - 1)}]$ are used in any $C_i$, and only vertices in these intervals are used to construct the subgraphs $E_j^{(1)}$ and $E_j^{(2)}$. Hence, each $C_i$ is disjoint from all subgraphs $E_j^{(1)}$ and $E_j^{(2)}$.

{\bf{Subcase 1.2:}}
Assume $m = 5$. Let $C_i=P_i \ \textcolor{blue}{x_{(\ell-1)+i}\  x_{2\ell+3i+1} \ x_{1-i}}.$
It can be seen that $C_i$ traverses edges of differences $ 2i-1,2i,\ell+2i-1, \textcolor{blue}{\ell+2i+2}, \textcolor{blue}{d}$.
Note that for $i\in\{1,2,\ldots,\floor{\frac{3k+1}{4}}\}$, the difference \textcolor{blue}{$d=2\ell+4i$}, and for $i\in\{\floor{\frac{3k+1}{4}}+1,\ldots,k\}$, the difference \textcolor{blue}{$d=5\ell-4i$}. 
\begin{eqnarray*}
C_1: &1,2, \ell+1, \textcolor{blue}{\ell+4}, \textcolor{blue}{2\ell+4}\\
C_2: &3,4, \ell+3, \textcolor{blue}{\ell+6}, \textcolor{blue}{2\ell+8}\\
C_3: &5,6, \ell+5, \textcolor{blue}{\ell+8}, \textcolor{blue}{2\ell+12}\\
&\vdots& \\
C_{k-2}: &2k-5,2k-4,2\ell-6, \textcolor{blue}{2\ell-3}, \textcolor{blue}{3\ell+10}\\
C_{k-1}: &2k-3,2k-2, 2\ell-4, \textcolor{blue}{2\ell-1}, \textcolor{blue}{3\ell+6}\\
C_k: &2k-1,2k, 2\ell-2, \textcolor{blue}{2\ell+1}, \textcolor{blue}{3\ell+2}
\end{eqnarray*}
For $m=5$, the missing differences are as follows:
\begin{enumerate}[(i)]
\item $2\ell+4j-2=4k+4j$, for $j\in\{1,2,\ldots,\floor{\frac{3}{4}(k+1)}\}$;
\item  $5\ell-4j+2=10k-4j+7$, for $j\in\{\floor{\frac{3}{4}(k+1)}+1, \ldots, k\}$; 
\item $2\ell+2j+1=4k+2j+3$, for $j\in\{1,2,\ldots,k-1\}$; and
\item $\ell+2=2k+3$.

\end{enumerate}
The following subgraphs cover the differences listed above: 
\begin{itemize}

\item Let \( E_j^{(1)} = x_{-(k+3j)}\, x_{3k+j} \); it is of difference \( 2\ell + 4j - 2 \) for \( j \in \left\{ 1, 2, \ldots, \left\lfloor \tfrac{3}{4}(k+1) \right\rfloor \right\} \), and of difference \( 5\ell - 4j + 2 \) for \( j \in \left\{ \left\lfloor \tfrac{3}{4}(k+1) \right\rfloor + 1, \ldots, k \right\} \);


\item Let \( E_j^{(2)} = x_{-(3k + j + 3)}\, x_{k + j} \); it is of difference \( 2\ell + 2j + 1 \) for \( j \in \{1, 2, \ldots, k - 1\} \); and


\item Let \( E_k^{(2)} = x_{-(7k + 3)}\, x_{-5k} \); it is of difference \( \ell + 2 \).


\end{itemize}
Observe that none of the vertices in the intervals \( [x_{-(k + 1)},\ x_{-(7k + 3)}] \), \( [x_{k + 1},\ x_{2k}] \), and \( [x_{3k + 1},\allowbreak\ x_{4k + 5}] \) are used in any \( C_i \). Only vertices from these intervals are used to construct the graphs \( E_j^{(1)} \) and \( E_j^{(2)} \). Hence, each \( C_i \) is disjoint from all subgraphs \( E_j^{(1)} \) and \( E_j^{(2)} \). 

\begin{figure}[hp]
 \centering
    \includegraphics[width=0.75\linewidth]{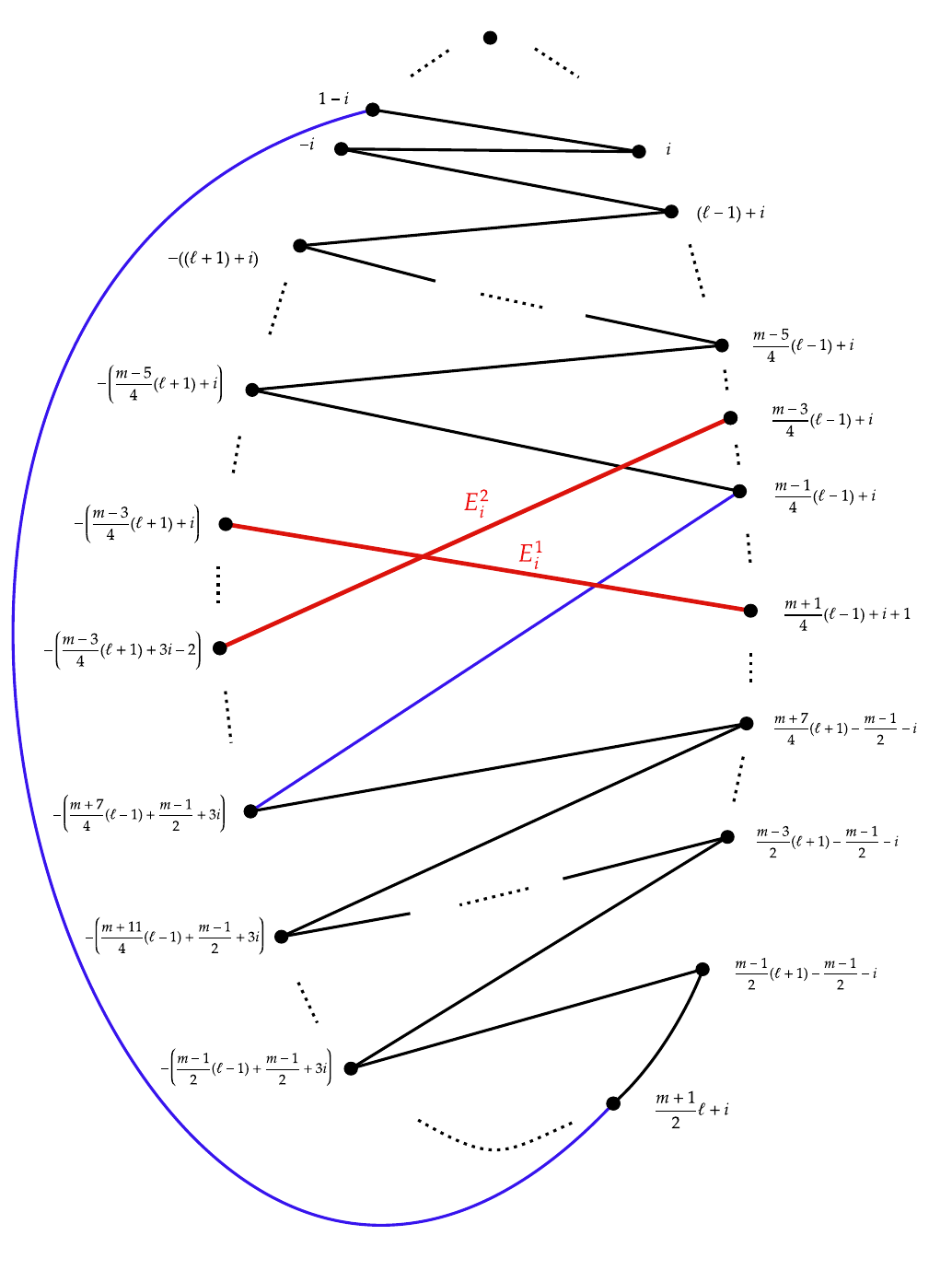}
    \caption{ Subgraph $F_i$ in Lemma~\ref{lem:Newtool-T2} for $m\equiv 1\ ({\rm mod}\ 4)$ and $m\geq 9$.}
    \label{fig:starter2}
  \end{figure}

\noindent {\bf{Case 2:}}  {\boldmath{$m\equiv 3\ ({\rm mod}\ 4)$.}}
 For $i\in\{1,2,\ldots,k\}$, define walks $P_i$ and $Q_i$ as follows:
\begin{eqnarray*}
P_i &=& x_{1-i}\ x_i\ x_{-i}\ x_{(\ell-1)+i}\ x_{-((\ell+1)+i)}\ x_{2(\ell-1)+i}\ x_{-(2(\ell+1)+i)} \ldots \\
&\ldots&   x_{\frac{m-11}{4}(\ell-1)+i}\ x_{-(\frac{m-11}{4}(\ell+1)+i)} \ x_{\frac{m-7}{4}(\ell-1)+i}\ x_{-(\frac{m-7}{4}(\ell+1)+i)} \ x_{\frac{m-3}{4}(\ell-1)+i}.
\\
Q_i &=& x_{-({\frac{m+5}{4}(\ell-1)+\frac{m+3}{2}-3i})}\ \ x_{\frac{m+5}{4}(\ell+1)-\frac{m+3}{2}+5i}\ \ x_{-({\frac{m+9}{4}(\ell-1)+\frac{m+3}{2}-3i})}\\ 
&& x_{\frac{m+9}{4}(\ell+1)-\frac{m+3}{2}+5i}\ \ \ x_{-({\frac{m+13}{4}(\ell-1)+\frac{m+3}{2}-3i})}\ \ x_{\frac{m+13}{4}(\ell+1)-\frac{m+3}{2}+5i}\ \ldots\\
&& \ldots  \ x_{-({\frac{m-5}{2}(\ell-1)+\frac{m+3}{2}-3i})}\ \ x_{{\frac{m-5}{2}(\ell+1)-\frac{m+3}{2}+5i}} \ \ x_{-({\frac{m-3}{2}(\ell-1)+\frac{m+3}{2}-3i})}\\
&& x_{{\frac{m-3}{2}(\ell+1)-\frac{m+3}{2}+5i}} \ \ x_{-({\frac{m-1}{2}(\ell-1)+\frac{m+3}{2}-3i})}\ \ x_{{\frac{m-1}{2}(\ell+1)-\frac{m+3}{2}+5i}}\ \ x_{-({\frac{m+1}{2}(\ell-1)+\frac{m+3}{2}-3i})}. 
\end{eqnarray*}
It is easy to verify that  $P_i$ is a path of length $\frac{m-1}{2}$ and $Q_i$ is a path of length  $\frac{m-3}{2}$. 
Now, for $i\in\{1,2,\ldots,k\}$, let: 
$$C_i=P_i \ \textcolor{blue}{x_{\frac{m-3}{4}(\ell-1)+i}\  x_{-({\frac{m+5}{4}(\ell-1)+\frac{m+3}{2}-3i})}}\ Q_i \ \textcolor{blue}{x_{-({\frac{m+1}{2}(\ell-1)+\frac{m+3}{2}-3i})} \ x_{1-i}}.$$
It can be seen that $C_i$ is a cycle of length $m$, and traverses edges of the following differences: 
\begin{eqnarray*}
&2i-1,2i, \ell+2i-1, 2\ell+2i,3\ell+2i-1, 4\ell+2i, 5\ell+2i-1, \ldots,\\
&  \frac{m-11}{2}\ell+2i, \frac{m-9}{2}\ell+2i-1, \frac{m-7}{2}\ell+2i, \frac{m-5}{2}\ell+2i-1,  \\
& \textcolor{blue}{\frac{m+1}{2}\ell-2i+1}, \frac{m-1}{2}\ell-2i,  \frac{m-3}{2}\ell-2i+1,  \frac{m-5}{2}\ell-2i, \ldots, \\
&5\ell-2i, 4\ell-2i+1, 3\ell-2i, 2\ell-2i+1, \textcolor{blue}{\frac{m+1}{2}\ell-4i+2}
\end{eqnarray*}
The explicit differences in $C_1,C_2, \ldots, C_k$ are listed below.
{\footnotesize{
\begin{eqnarray*}
C_1: &1,2, \ell+1,\ldots,  \frac{m-5}{2}\ell+1,\textcolor{blue}{\frac{m+1}{2}\ell-1}, \frac{m-1}{2}\ell-2, \ldots,3\ell-2, 2\ell-1, \textcolor{blue}{\frac{m+1}{2}\ell-2}\\
C_2: &3,4, \ell+3,\ldots,   \frac{m-5}{2}\ell+3, \textcolor{blue}{\frac{m+1}{2}\ell-3},  \frac{m-1}{2}\ell-4, \ldots,3\ell-4, 2\ell-3, \textcolor{blue}{\frac{m+1}{2}\ell-6}\\
C_3: &5,6, \ell+5,\ldots, \frac{m-5}{2}\ell+5, \textcolor{blue}{\frac{m+1}{2}\ell-5}, \frac{m-1}{2}\ell-6, \ldots,3\ell-6, 2\ell-5, \textcolor{blue}{\frac{m+1}{2}\ell-10}\\
&\vdots& \\
C_{k-2}: &2k-5,2k-4,2\ell-6,\ldots,  \frac{m-3}{2}\ell-6, \textcolor{blue}{\frac{m-1}{2}\ell+6},  \frac{m-3}{2}\ell+5, \ldots,2\ell+5, \ell+6, \textcolor{blue}{\frac{m-3}{2}\ell+12}\\
C_{k-1}: &2k-3,2k-2, 2\ell-4, \ldots,\frac{m-3}{2}\ell-4, \textcolor{blue}{\frac{m-1}{2}\ell+4}, \frac{m-3}{2}\ell+3, \ldots,2\ell+3, \ell+4, \textcolor{blue}{\frac{m-3}{2}\ell+8}\\
C_k: &2k-1,2k, 2\ell-2,\ldots, \frac{m-3}{2}\ell-2, \textcolor{blue}{\frac{m-1}{2}\ell+2},  \frac{m-3}{2}\ell+1, \ldots,2\ell+1, \ell+2, \textcolor{blue}{\frac{m-3}{2}\ell+4}
\end{eqnarray*}}}
Observe that $C_1, C_2, \ldots, C_k$ jointly contain exactly one edge of each difference in $S$  except for the following differences:
\begin{enumerate}[(i)]
\item $\frac{m-3}{2}\ell+4j-2$, for $j\in\{1,2,\ldots,k\}$;
\item $\frac{m+1}{2}\ell+j$, for $j\in\{1,2,\ldots,k\}$.
\end{enumerate}
Note that we have $2k$ differences from $S$ that are not covered by any $C_i$. 

{\bf{Subcase 2.1:}} Assume $m\geq  7$. 
For all $j\in\{1,2,\ldots,k\}$, let $E_j^{(1)}$ and $E_j^{(2)}$ denote the following subgraphs of $K_n$ (isomorphic to $K_2$):
\begin{itemize}
\item $E_j^{(1)}= x_{-(\frac{m-5}{4}(\ell+1)+2j-1)}\ \ x_{\frac{m-1}{4}(\ell-1)+2j}$; and
\item $E_j^{(2)}= x_{-(\frac{m-1}{4}(\ell+1)+1)}\ \ x_{\frac{m+3}{4}(\ell-1)+j}$.
\end{itemize}
Note that, for $j\in\{1,2,\ldots,k\}$, the edge in $E_j^{(1)}$ is of difference  $\frac{m-3}{2}\ell+4j-2$, and the edge in $E_j^{(2)}$ is of difference  $\frac{m+1}{2}\ell+j$. Moreover, none of the vertices in the intervals $[x_{-(\frac{m-5}{4}(\ell+1))}, x_{-(\frac{m-1}{4}(\ell+1)+1)}]$ and  $[x_{\frac{m-1}{4}(\ell-1)+1}, x_{\frac{m+5}{4}(\ell-1)+5}]$ are used in any $C_i$, and only vertices from these intervals are used  to construct  the subgraphs $E_j^{(1)}$ and $E_j^{(2)}$.  Hence, each $C_i$ is disjoint from all subgraphs $E_j^{(1)}$ and $E_j^{(2)}$. 

{\bf{Subcase 2.2:}} Assume $m=3$. The missing differences are as follows:
\begin{enumerate}[(i)]
\item $4j-2$, for $j\in\{1,2,\ldots,k\}$;
\item  $2\ell+j=4k+j+2$, for $j\in\{1,2,\ldots,k\}$.
\end{enumerate}
The following subgraphs cover the differences listed above: 
\begin{itemize}
\item the graph \( E_j^{(1)} = x_{3k + 2j}\, x_{3k - 2j + 2} \), which is of difference \( 4j - 2 \); and

\item the graph $E_j^{(2)}=x_{-(k+2)} \ x_{3k+j}$, which is of difference $4k+j+2$.

\end{itemize}
Observe that the vertex \( x_{-(k + 2)} \), as well as all vertices in the interval \( [x_{k + 1},\ x_{5k + 2}] \), are not used in any \( C_i \). Moreover, only vertices from this interval are used to construct the graphs \( E_j^{(1)} \) and \( E_j^{(2)} \). Hence, each \( C_i \) is disjoint from all subgraphs \( E_j^{(1)} \) and \( E_j^{(2)} \).


In both cases above, we have shown how to construct \( m \)-cycles \( C_1, C_2, \ldots, C_k \), and subgraphs \( E_1^{(1)}, E_2^{(1)}, \ldots, E_k^{(1)} \) and \( E_1^{(2)}, E_2^{(2)}, \ldots, E_k^{(2)} \). Now, for \( i \in \{1, 2, \ldots, k\} \), let \( F_i = C_i \mathbin{\dot{\cup}} \left( E_i^{(1)} \cup E_i^{(2)} \right) \). Since \( F_1, F_2, \ldots, F_k \) jointly contain exactly one edge of each difference in \( S \), we see that
\(
\mathcal{D} = \{ \rho^i(F_1), \rho^i(F_2), \ldots, \rho^i(F_k) : i \in \mathbb{Z}_n \}
\)
is a decomposition of \( K_{\ell[m+2]} \) that satisfies the conditions of Lemma~\ref{lem:Newtool-00}. Applying Lemma~\ref{lem:Newtool-00}, we conclude that the multigraph \( 4K_{\ell[m+2]}^{\bullet} \) admits an HOP \((C_2, C_m)\)-decomposition.
\end{proof}


\subsection{The Main Results for Two Round Tables}


Combining the results of Lemmas~\ref{lem:Newtool-1b} and~\ref{lem:Newtool-T2} with the previous results in Theorem~\ref{theo:(Cm1,Cm2)-dec}, we prove the following main results. We shall restate the main theorem for ease of reference.

\mainresultOne*
\begin{proof}
{\bf{(i)}}  Since $n \equiv 1 \pmod{2m}$, we have $n = 2mk + 1$ for some integer $k \geq 1$. There are three cases to consider.
First, if $m_1 = m_2 = 2$, then $n = 8k + 1$. By Theorem~\ref{theo:Cm-dec}, there exists a $(C_4)$-decomposition of $K_n$. Applying Corollary~\ref{lem:Newtool-01}, it follows that $4K_n^{\bullet}$ admits an HOP $(C_2, C_2)$-decomposition. 
Second, if $m_1 = 2$ and $m_2 \geq 3$, then by Lemma~\ref{lem:Newtool-1b}, the multigraph $4K_n^{\bullet}$ admits an HOP $(C_2, C_{m_2})$-decomposition. 
Third, if $m_1 \geq 3$, then by Theorem~\ref{theo:(Cm1,Cm2)-dec}(i), there exists a $(C_{m_1}, C_{m_2})$-decomposition of $K_n$. Applying Lemma~\ref{lem:Gtool4}, it follows that $4K_n^{\bullet}$ admits an HOP $(C_{m_1}, C_{m_2})$-decomposition. 

\noindent {\bf{(ii)}}  We have two cases to consider. First, if $m_1 = 2$ and $m_2 \geq 3$, then since $m = m_1 + m_2$ is odd, $m_2$ must be odd. It follows from Lemma~\ref{lem:Newtool-T2} that the multigraph $4K_n^{\bullet}$ admits an HOP $(C_2, C_{m_2})$-decomposition. 
Second, if $m_1 \geq 3$, we have two subcases; if $(m_1, m_2, n) \neq (4,5,9)$, then by Theorem~\ref{theo:(Cm1,Cm2)-dec}(ii), there exists a $(C_{m_1}, C_{m_2})$-decomposition of $K_n$, and applying Lemma~\ref{lem:Gtool4}, the multigraph $4K_n^{\bullet}$ admits an HOP $(C_{m_1}, C_{m_2})$-decomposition.  If $(m_1, m_2, n) = (4,5,9)$, then by Lemma~\ref{lems-1starter}(i), an HOP $(C_4, C_5)$-factorization of $4K_9^{\bullet}$ exists. 

In both cases, we see that $4K_n^{\bullet}$ admits an HOP $(C_{m_1}, C_{m_2})$-decomposition. Therefore, by Theorem~\ref{theo:Gtool1}, $\mathrm{HOP}(2^{\langle s \rangle}, 2m_1, 2m_2)$ has a solution.
\end{proof}

%

An easy generalization of the $(C_2, C_m)$-decompositions constructed in Lemmas~\ref{lem:Newtool-1b} and~\ref{lem:Newtool-T2} is to include multiple copies of $C_2$ in the decomposition. This is described below.

\begin{lem}{\label{cor:HOP-decomp}}
Let \( r \geq 1 \), and \( m \geq 3 \) be integers. Then, there exists an HOP $(C_2^{\langle r \rangle}, C_m)$-decomposition of $4K_{n}^{\bullet}$ in each of the following cases:

\begin{enumerate}[\bf(i)]
    \item  \( n \equiv 1 \pmod{2(2r + m)} \)
    \item   $m$ is odd and $n \equiv 2r + m \pmod{2(2r + m)}$
 \end{enumerate}   
\end{lem}
\begin{proof}
\noindent {\bf{(i)}}
Assume $n \equiv 1 \pmod{2(m + 2r)}$. 
By Lemma~\ref{lem:Newtool-1b}, there exists an HOP $(C_2, C_m)$-decomposition of $4K_n^{\bullet}$. Hence, we may assume $r \geq 2$, then since $n \equiv 1 \pmod{2(2r+m)}$, by Theorem~\ref{theo:(Cm1,Cm2)-dec}(i) there exists a $(C_{2r}, C_m)$-decomposition of $K_n$. Applying Corollary~\ref{lem:Newtool-01}, we obtain an HOP $(C_2^{\langle r \rangle}, C_m)$-decomposition of $4K_n^{\bullet}$. 

\vspace{0.2cm}

\noindent {\bf{(ii)}} Assume $m$ is odd and $n \equiv 2r+m  \pmod{2(2r+m)}$. Then $n = (2k + 1)(2r+m)$ for some integer $k \geq 0$.

First, assume $k = 0$; then $n = m + 2r$.  We have three cases to consider. If $r = 1$, then by Lemma~\ref{lems-1starter}, the multigraph $4K_{m+2}^{\bullet}$ admits an HOP $(C_2, C_m)$-factorization. Second, if $r = 2$ and $m = 5$, then by Lemma~\ref{lems-1starter}(i), the multigraph $4K_9^{\bullet}$ admits an HOP $(C_2^{\langle 2 \rangle}, C_5)$-factorization. Third, if $r \geq 2$ and $m \neq 5$, then by Theorem~\ref{theo:OP4}, there exists a $2$-factorization of $K_{m+2r}$ into $(C_{2r}, C_m)$-factors. Applying Corollary~\ref{lem:Newtool-01}, we see that the multigraph $4K_{m+2r}^{\bullet}$ admits an HOP $(C_2^{\langle r \rangle}, C_m)$-factorization.

Now, assume $k \geq 1$. There are three cases to consider. If $r = 1$, then $n = (2k+1)(m+2)$, and by Lemma~\ref{lem:Newtool-T2}, there exists an HOP $(C_2, C_m)$-decomposition of $4K_n^{\bullet}$. Second, if $r = 2$ and $(m,n) = (5,9)$, then by Lemma~\ref{lems-1starter}(i), the multigraph $4K_9^{\bullet}$ admits an HOP $(C_2^{\langle 2 \rangle}, C_5)$-factorization. Third, if $r \geq 2$ and $(m,n) \neq (5,9)$, then by Theorem~\ref{theo:(Cm1,Cm2)-dec}(ii), the complete graph $K_n$ admits a $(C_{2r}, C_m)$-decomposition. Applying Corollary~\ref{lem:Newtool-01}, it follows that there exists an HOP $(C_2^{\langle r \rangle}, C_m)$-decomposition of $4K_n^{\bullet}$.    
\end{proof}

The following corollary is a direct consequence of Lemma~\ref{cor:HOP-decomp} and Theorem~\ref{theo:Gtool1}.

\begin{cor}{\label{corlem:HOP-decomp}}
Let \( s \geq 0 \), \( r \geq 1 \), and \( m \geq 3 \) be integers, and let \( n = s + 2r + m \). Then \( \mathrm{HOP}(2^{\langle s \rangle}, 4^{\langle r \rangle}, 2m) \) has a solution in each of the following cases:
\begin{enumerate}[\bf(i)]
    \item  \( n \equiv 1 \pmod{2(2r + m)} \)
    \item   $m$ is odd and $n \equiv 2r + m \pmod{2(2r + m)}$
 \end{enumerate}   
\end{cor}


\section{The Generalized HOP with Small Round Tables}{\label{sec:6}}

We restate our second main theorem here for ease of reference. The proof follows from existing work on decompositions of $K_n$ into $(C_{m_1}, C_{m_2}, \ldots, C_{m_t})$-subgraphs for $m_1, m_2, \ldots, m_t \geq 3$. However, we develop new tools to handle the cases where $2 \in \{m_1, m_2, \ldots, m_t\}$.
We proceed case by case, depending on the value of $m = m_1 + \cdots + m_t$, where $2 \leq m_1 \leq \cdots \leq m_t$ and $m \leq 10$.

\mainresultTwo*
\begin{proof} 
The proof is summarized in Tables~\ref{tab:version1} and~\ref{tab:version2}.
Note that for the cases highlighted in pink, Theorems~\ref{theo:Cm-dec} and \ref{theo:(Cm1,...,Cmt)-dec} are used to show that \( K_n \) admits a \( (C_{m_1}, C_{m_2}, \ldots, C_{m_t}) \)-decomposition. For these cases, we then apply Theorem~\ref{lem:Gtool4} to show that \( 4K_n^{\bullet} \) admits an HOP \( (C_{m_1}, C_{m_2}, \ldots, C_{m_t}) \)-decomposition. 

For the remaining cases, we combine the results Lemmas~\ref{cor:HOP-decomp} and \ref{lem:m=46-Samll0}--\ref{lem:m=10-Samll5} to show that \( 4K_n^{\bullet} \) admits an HOP \( (C_{m_1}, C_{m_2}, \ldots, C_{m_t}) \)-decomposition.

Tables~\ref{tab:version1} and \ref{tab:version2} list all possible cases where \( n(n - 1) \equiv 0 \pmod{2m} \) with \( m \leq 10 \) and \( n \) odd.  
Observe that in each case presented in the table, the multigraph \( 4K_n^{\bullet} \) admits an HOP \( (C_{m_1}, C_{m_2}, \ldots, C_{m_t}) \)-decomposition. Applying Theorem~\ref{theo:Gtool1}, we conclude that HOP \( (2^{\langle s \rangle}, 2m_1, \allowbreak \ldots, 2m_t) \) has a solution. 

Note that in the tables, when a row of the 3rd and 4th columns straddles two rows of column~2, then these decompositions may fall under both cases.
\end{proof}

\begin{table}[p]
            \small
            \renewcommand{\arraystretch}{1.37} 
            \begin{center}
              \begin{tabular}{V{3.5} p{0.33cm} V{3.5} p{2.3cm} V{4.5} p{3.5cm} V{3.5} p{7.9cm} V{3.5}}   
\Xhline{1.5pt}
\rowcolor{Gray}
 {\boldmath $m$} & \centering {\boldmath $n$} & \centering {\boldmath$ ({m_1},\ldots, {m_t})$} & \textbf{{\boldmath$4K_n^{\bullet}$} admits an HOP {\boldmath$(C_{m_1},\ldots, C_{m_t})$}\textbf{-decomposition} by}\\ \cline{1-4}
 \Xhline{1.5pt}  
 \multirow{1}{*}{2}
  & \multirow{1}{*}{$4k+1$, $k\geq 1$} & \centering $(2)$ & Lemma \ref{G4C} with $t=1$  \\ \cline{2-4}
 \Xhline{1.5pt} 
\multirow{1}{*}{3}
  & \multirow{1}{*}{$6k+1$, $k\geq 1$ } & \centering \cellcolor{PINK}{$(3)$}& \cellcolor{PINK}{Theorems~\ref{theo:Cm-dec}} and \ref{lem:Gtool4} \\  \clineB{2-2}{2}
\multirow{1}{*}{}
  & \multirow{1}{*}{$6k+3$, $k\geq 0$ } & \cellcolor{PINK}{ }& \cellcolor{PINK}{}  \\ \cline{2-4}
  \Xhline{1.5pt}
  \multirow{3}{*}{4} & \multirow{3}{*}{$8k+1$, $k\geq 1$} &\centering $(2, 2)$  & Lemma \ref{lem:m=46-Samll0}(i)   \\ \hhline{~~--}
   \multirow{3}{*}{} & \multirow{3}{*}{}  & \centering \cellcolor{PINK}{$(4)$}  &\cellcolor{PINK}{Theorems \ref{theo:Cm-dec}  and \ref{lem:Gtool4}}  \\  \cline{3-4}
 \Xhline{1.5pt}  
  \multirow{3}{*}{} & \multirow{3}{*}{$10k+1$, $k\geq 1$} & \centering $(2, 3)$  & Theorem~\ref{cor:HOP-decomp}(i) with $m=3$, $r=1$  \\ \hhline{~~--}
  \multirow{3}{*}{5} & \multirow{3}{*}{}& \centering \cellcolor{PINK}{ $(5)$}  & \cellcolor{PINK}{ Theorems \ref{theo:Cm-dec} and \ref{lem:Gtool4}} \\ \clineB{2-2}{3.5}
  \multirow{3}{*}{} & \multirow{3}{*}{$10k+5$, $k\geq 0$} & \cellcolor{PINK}{} &\cellcolor{PINK}{}  \\ \cline{3-4}
  \multirow{3}{*}{} & \multirow{3}{*}{}&\centering$(2, 3)$  & Theorem~\ref{cor:HOP-decomp}(ii) with $m=3$, $r=1$   \\ \cline{2-4}
 \Xhline{1.5pt}
  \multirow{3}{*}{} & \multirow{3}{*}{$12k+1$, $k\geq 1$} & \centering $(2, 2, 2)$  & Lemma \ref{lem:m=46-Samll0}(ii)   \\ \cline{3-4}
   \multirow{3}{*}{6} & \multirow{3}{*}{}  & \centering $(2, 4)$  & Theorem~\ref{cor:HOP-decomp}(i) with $m=4$, $r=1$  \\  \hhline{~~--}
  \multirow{3}{*}{} & \multirow{3}{*}{}&  \centering \cellcolor{PINK}{ $(3, 3)$}  &  \cellcolor{PINK}{Theorem \ref{theo:(Cm1,...,Cmt)-dec} if $n\neq9$ and Theorem~\ref{lem:Gtool4} }  \\ \clineB{2-2}{3.5}
  \multirow{3}{*}{} & \multirow{3}{*}{$12k+9$, $k\geq 0$} &\centering  \cellcolor{PINK}{$(6)$} &  \cellcolor{PINK}{ Theorems \ref{theo:Cm-dec} and \ref{lem:Gtool4}}\\ \cline{3-4}
   \multirow{3}{*}{} & \multirow{3}{*}{}&\centering $(3, 3)$  & Lemma \ref{lem:m=6-Samll2} if $n=9$ \\ \cline{3-4}
  \multirow{3}{*}{} & \multirow{3}{*}{}& \centering $(2, 2, 2)$  & Lemma \ref{lem:m=46-Samll0}(ii)    \\ \cline{3-4} 
  \multirow{3}{*}{} & \multirow{3}{*}{}&\centering $(2, 4)$  & Lemma \ref{lem:m=6-Samll1}  \\ \cline{2-4} 
 \Xhline{1.5pt}
  \multirow{3}{*}{} & \multirow{3}{*}{$14k+1$, $k\geq 1$} & \centering $(2, 2, 3)$  & Theorem~\ref{cor:HOP-decomp}(i) with $m=3$, $r=2$   \\ \cline{3-4}
   \multirow{3}{*}{7} & \multirow{3}{*}{}  & \centering $(2, 5)$  & Theorem~\ref{cor:HOP-decomp}(i) with $m=5$, $r=1$  \\  \hhline{~~--}
  \multirow{3}{*}{} & \multirow{3}{*}{}&  \centering \cellcolor{PINK}{$(3, 4)$}  &  \cellcolor{PINK}{Theorems \ref{theo:(Cm1,...,Cmt)-dec} and \ref{lem:Gtool4}}  \\ \clineB{2-2}{3.5}
  \multirow{3}{*}{} & \multirow{3}{*}{$14k+7$, $k\geq 0$} & \centering \cellcolor{PINK}{$(7)$} &  \cellcolor{PINK}{ Theorems \ref{theo:Cm-dec} and \ref{lem:Gtool4}}  \\ \cline{3-4}
  \multirow{3}{*}{} & \multirow{3}{*}{}& \centering $(2, 2, 3)$  & Theorem~\ref{cor:HOP-decomp}(ii) with $m=3$, $r=2$   \\ \cline{3-4} 
  \multirow{3}{*}{} & \multirow{3}{*}{}&\centering $(2, 5)$  & Theorem~\ref{cor:HOP-decomp}(ii) with $m=5$, $r=1$  \\ \cline{2-4}
 \Xhline{1.5pt}
  \multirow{3}{*}{} & \multirow{3}{*}{} &\centering $(2, 2, 2, 2)$  & Lemma \ref{lem:m=8-Samll0}  \\ \cline{3-4}
  \multirow{3}{*}{8} & \multirow{3}{*}{$16k+1$, $k\geq 1$}  & \centering $(2, 2,4)$  &Theorem~\ref{cor:HOP-decomp}(i) with $m=4$, $r=2$  \\  \cline{3-4} 
  \multirow{3}{*}{} & \multirow{3}{*}{}  & \centering $(2, 3,3)$  &Lemma \ref{lem:m=8-Samll1}   \\  \cline{3-4} 
  \multirow{3}{*}{} & \multirow{3}{*}{}  & \centering $(2, 6)$  &Theorem~\ref{cor:HOP-decomp}(i) with $m=6$, $r=1$  \\  \hhline{~~--}
  \multirow{3}{*}{} & \multirow{3}{*}{}  & \centering \cellcolor{PINK}{$(3, 5)$, $(4, 4)$}  & \cellcolor{PINK}{Theorems \ref{theo:(Cm1,...,Cmt)-dec} and ~\ref{lem:Gtool4} }   \\  \cline{3-4} 
  \multirow{3}{*}{} & \multirow{3}{*}{}  &\centering  \cellcolor{PINK}{$(8)$}  & \cellcolor{PINK}{Theorems \ref{theo:Cm-dec} and ~\ref{lem:Gtool4} }   \\  \cline{3-4} 
 \Xhline{1.5pt}    
\end{tabular}
                \caption{HOP $(C_{m_1},\ldots, C_{m_t})$-decompositions with $m_1+\ldots+m_t\leq 8$}
                \label{tab:version1}
            \end{center}
\end{table}

\begin{table}[p]
            \small
            \renewcommand{\arraystretch}{1.37} 
            \begin{center}
              \begin{tabular}{V{3.5} p{0.33cm} V{3.5} p{2.3cm} V{4.5} p{3.5cm} V{3.5} p{7.9cm} V{3.5}}   
\Xhline{1.5pt}
\rowcolor{Gray}
 {\boldmath $m$} & \centering {\boldmath $n$} & \centering {\boldmath$ ({m_1},\ldots, {m_t})$} & \textbf{{\boldmath$4K_n^{\bullet}$} admits an HOP {\boldmath$(C_{m_1},\ldots, C_{m_t})$}\textbf{-decomposition} by}\\ \cline{1-4}
 \Xhline{1.5pt}  
  \multirow{3}{*}{} & \multirow{3}{*}{$18k+1$, $k\geq 1$} & \centering $(2, 2, 2, 3)$  & Theorem~\ref{cor:HOP-decomp}(i) with $m=3$, $r=3$   \\ \cline{3-4}
  \multirow{3}{*}{} & \multirow{3}{*}{} & \centering $(2, 2, 5)$  & Theorem~\ref{cor:HOP-decomp}(i) with $m=5$, $r=2$   \\ \cline{3-4}
  \multirow{3}{*}{} & \multirow{3}{*}{} & \centering $(2, 3, 4)$  & Lemma \ref{lem:m=9-Samll0}   \\ \cline{3-4}
   \multirow{3}{*}{} & \multirow{3}{*}{} & \centering $(2, 7)$  & Theorem~\ref{cor:HOP-decomp}(i) with $m=7$, $r=1$   \\ \hhline{~~--}
  \multirow{3}{*}{9} & \multirow{3}{*}{}&\centering  \cellcolor{PINK}{$(3, 3, 3)$, $(3, 6)$, $(4, 5)$}  & \cellcolor{PINK}{Theorem \ref{theo:(Cm1,...,Cmt)-dec} {\footnotesize{(except for  $(C_4, C_5)$-decomp when $n= 9$)}} and Theorem~\ref{lem:Gtool4} } \\ \clineB{2-2}{3.5}
  \multirow{3}{*}{} & \multirow{3}{*}{$18k+9$, $k\geq 0$} & \centering \cellcolor{PINK}{$(9)$} & \cellcolor{PINK}{Theorems \ref{theo:Cm-dec} and \ref{lem:Gtool4}} \\ \cline{3-4}
  \multirow{3}{*}{} & \multirow{3}{*}{}&\centering $(4, 5)$  & Lemma \ref{lems-1starter}(i) if $n= 9$ \\ \cline{3-4}
  \multirow{3}{*}{} & \multirow{3}{*}{} & \centering $(2, 2, 2,  3)$  & Theorem~\ref{cor:HOP-decomp}(ii) with $m=3$, $r=3$   \\ \cline{3-4}
  \multirow{3}{*}{} & \multirow{3}{*}{} & \centering $(2, 2, 5)$  & Theorem~\ref{cor:HOP-decomp}(ii) with $m=5$, $r=2$   \\ \cline{3-4} 
  \multirow{3}{*}{} & \multirow{3}{*}{} &\centering $(2, 3, 4)$  & Lemma \ref{lem:m=9-Samll1}   \\ \cline{3-4}
  \multirow{3}{*}{} & \multirow{3}{*}{}&\centering $(2, 7)$  & Theorem~\ref{cor:HOP-decomp}(ii) with $m=7$, $r=1$  \\ \cline{2-4} 
 \Xhline{1.5pt}
  \multirow{3}{*}{} & \multirow{3}{*}{} & \centering $(2, 2, 2, 2, 2)$  & Lemma \ref{lem:m=10-Samll0}(\ref{lem:m=10-Samll0-1})   \\ \cline{3-4}
 \multirow{3}{*}{} & \multirow{3}{*}{$20k+1$, $k\geq 1$} & \centering $(2, 2, 2, 4)$  & Theorem~\ref{cor:HOP-decomp}(i) with $m=4$, $r=3$   \\ \cline{3-4}
 \multirow{3}{*}{} & \multirow{3}{*}{} & \centering $(2, 2, 3, 3)$  & Lemma \ref{lem:m=10-Samll0}(\ref{lem:m=10-Samll0-3}) \\ \cline{3-4}
 \multirow{3}{*}{} & \multirow{3}{*}{} & \centering $(2, 2, 6)$  & Theorem~\ref{cor:HOP-decomp}(i) with $m=6$, $r=2$   \\ \cline{3-4}
 \multirow{3}{*}{} & \multirow{3}{*}{} & \centering $(2, 3, 5)$  & Lemma \ref{lem:m=10-Samll2}  \\ \cline{3-4}
 \multirow{3}{*}{} & \multirow{3}{*}{} & \centering $(2, 4, 4)$  & Lemma \ref{lem:m=10-Samll1}   \\ \cline{3-4}
 \multirow{3}{*}{} & \multirow{3}{*}{} & \centering $(2, 8)$  & Theorem~\ref{cor:HOP-decomp}(i) with $m=8$, $r=1$   \\ \hhline{~~--} 
  \multirow{3}{*}{10} & \multirow{3}{*}{}& \centering \cellcolor{PINK}{  $(3, 3, 4)$, $(3, 7)$, $(4, 6)$, $(5, 5)$}  &  \cellcolor{PINK}{Theorems \ref{theo:(Cm1,...,Cmt)-dec} and \ref{lem:Gtool4} }  \\ \clineB{2-2}{3.5}
   \multirow{3}{*}{} & \multirow{3}{*}{} & \centering \cellcolor{PINK}{$(10)$}  & \cellcolor{PINK}{Theorems \ref{theo:Cm-dec} and \ref{lem:Gtool4}}  \\ \cline{3-4}
  \multirow{3}{*}{} & \multirow{3}{*}{$20k+5$, $k\geq 1$} &\centering  $(2, 2, 2, 2, 2)$, $(2, 2, 2, 4)$, $(2, 2, 3, 3)$, $(2, 2, 6)$  & Lemma \ref{lem:m=10-Samll0}    \\ \cline{3-4}
 \multirow{3}{*}{} & \multirow{3}{*}{} & \centering $(2, 3, 5)$  & Lemma \ref{lem:m=10-Samll5}    \\ \cline{3-4}
 \multirow{3}{*}{} & \multirow{3}{*}{} & \centering $(2, 4, 4)$  & Lemma \ref{lem:m=10-Samll4}   \\ \cline{3-4}
 \multirow{3}{*}{} & \multirow{3}{*}{} & \centering $(2, 8)$  & Lemma \ref{lem:m=10-Samll3}  \\ \cline{2-4} 
 \Xhline{1.5pt}
 
 \end{tabular}
                \caption{HOP $(C_{m_1},\ldots, C_{m_t})$-decompositions with $m_1+\ldots+m_t\in \{9,10\}$}
                \label{tab:version2}
            \end{center}
\end{table}
\subsection{HOP Decompositions for $n\equiv 1\ ({\rm mod}\ 8)$ and $n\equiv 1 \text{ \normalfont{or} } 9\ ({\rm mod}\ 12)$}

\begin{lem} \label{lem:m=46-Samll0}
The following decompositions of $4K_n^{\bullet}$ exist:
\begin{enumerate}[\bf(i)]
    \item an HOP $(C_2, C_2)$-decomposition when $n \equiv 1 \pmod{8}$;
    \item an HOP $(C_2, C_2, C_2)$-decomposition when $n \equiv 1, 9 \pmod{12}$.
\end{enumerate}
\end{lem}

\begin{proof}
By Theorem~\ref{theo:Cm-dec}, there exists a $(C_4)$-decomposition of $K_n$ when $n \equiv 1 \pmod{8}$ and a $(C_6)$-decomposition of $K_n$ when $n \equiv 1, 9 \pmod{12}$. The results follow by applying Corollary~\ref{lem:Newtool-01}. 
\end{proof}

\begin{lem}{\label{lem:m=6-Samll1}}
Let $n=12k+9$ for any integer $k\geq 0$. There exists an HOP $(C_2,C_4)$-decomposition of $4K_n^{\bullet}$. 
\end{lem}
\begin{proof}
Let $s=4k+3$. By Lemma \ref{lem:4Kn-d} with $m_1=2$ and $m_2=4$, it suffices to find $(C_2,C_4)$-subgraphs $F_1,F_2,\ldots, F_s$ in $4K_n^{\bullet}$ that satisfy Conditions (B1) to (B3).

As in Lemma \ref{lem:4Kn-d}, let $V(4K_n^{\bullet})=\{x_i: i\in \ZZ_{n-1} \}\cup \{x_{\infty}\}$.  Let $\rho_{\bullet}$ be the permutation on $E(4K_n^{\bullet})$ that preserves the color (and orientation) of the edges, and is induced by  the permutation  $\rho=(x_{\infty})(x_0 \ x_1\  x_2 \ \ldots \ x_{n-2})$. Observe that the group $\langle \rho_{\bullet} \rangle$ has the following orbits on the edge set of $4K_n^{\bullet}$:
\begin{itemize}
\item for each $d\in \{1,2,\ldots, 6k+3\}$, we have a pink and a blue orbit $\{x_ix_{i+d}: i\in \ZZ_{n-1} \}$;
\item for each $d\in \{1,2,\ldots, 12k+7\}$, we have a black orbit $\{(x_i, x_{i+d}): i\in \ZZ_{n-1} \}$;
\item a pink and a blue orbit $\{ x_ix_{i+(6k+4)}: i=0,1,\ldots,\frac{n-3}{2}\}$;
\item a pink and a blue orbit  $\{x_ix_{\infty}: i\in \ZZ_{n-1} \}$; and
\item black orbits $\{(x_{\infty}, x_{i}): i\in \ZZ_{n-1} \}$ and $\{(x_i, x_{\infty}): i\in \ZZ_{n-1} \}$.
\end{itemize}
For convenience, let \( S = \left\{1, 2, \ldots, 6k+4, \infty \right\} \) be the set of all differences.

We start by constructing subgraphs \( F_1, F_2, \ldots, F_{s-1} \), where each \( F_i \) is a disjoint union of a 4-cycle and a 2-cycle, and the subgraphs \( F_1, F_2, \ldots, F_{s-1} \) jointly cover exactly four edges corresponding to the following differences from \( S \):
$$  1, 2, \ldots, 3k+1, 3k+3, \ldots,  6k+3,  \infty. $$ 

First, we construct the 4-cycles. For \( i \in \{1, 2, \ldots, 2k + 1\} \), let \( C_i \) be the following closed walk:
\[
C_i = x_{0} \ x_{i} \ x_{6k+4} \ x_{-(6k+4 - i)} \ x_{0}.
\]
It can be seen that \( C_i \) is a cycle of length four and traverses edges of differences \( i \) and \( 6k + 4 - i \) twice. This means that for \( i \in \{1, 2, \ldots, 2k + 1\} \), the cycles \( C_i \)  
jointly cover each of the following differences exactly twice:
\[
1, 2, \ldots, 2k + 1, \; 4k + 3, 4k + 4, \ldots, 6k + 2, 6k + 3.
\]
%
%
%
%
%
%
Next, construct cycles of length two, denoted by \( E_i \), as follows:
\begin{itemize}
\item for $i\in\{1,2,\ldots, 2k+1\}$ and $i\neq k+1$, let $E_i=x_{3k+2}\ x_{5k+3+i}\ x_{3k+2}$; and
\item for $i=k+1$, let $E_i=x_{3k+2}\ x_{\infty}\ x_{3k+2}$.
\end{itemize}
Observe that the 2-cycles \( E_i \), for \( i \in \{1, 2, \ldots, 2k + 1\} \), jointly cover each of the following differences exactly twice:
\[
2k + 2, 2k + 3, \ldots, 3k + 1, \; \infty, \; 3k + 3, \ldots, 4k + 1, 4k + 2.
\]
%
It can be easily checked that, for $i\in\{1,2,\ldots, 2k+1\}$, each $C_i$ is disjoint from each $E_i$. 
Now, for $i\in\{1,2,\ldots,2k+1\}$, let $G_i=C_i\mathbin{\dot{\cup}} E_i$. We see that $G_1,G_2, \ldots, G_{2k+1}$ jointly contain exactly two edges of each of the differences in $S$, except for differences $3k+2$ and $6k+4$. 

Consider two copies of each $G_i$, for all $i\in\{1,2,\ldots,2k+1\}$. We want to color each copy so that it satisfies Condition {\bf (C1)} of Definition {\rm{\ref{def}}}. In the first copy of $G_i$, color the edges of $E_i$ pink and blue; and in $C_i$, color the edge $x_{0}\ x_{i}$ pink, $x_{6k+4}\ x_{-(6k+4-i)}$ blue, and color the other two edges black and orient them away from the pink edge and towards the blue edge. In the second copy of $G_i$,  color the edges of $E_i$ black and orient them in opposite directions; and in $C_i$, color $x_{i}\ x_{6k+4}$ pink, $x_{-(6k+4-i)}\ x_{0}$ blue, and color the other two edges black and orient them away from the pink edge and towards the blue edge.  Considering two copies of each $G_i$, for all $i\in\{1,2,\ldots,2k+1\}$, we have $4k+2$ $(C_2,C_4)$-subgraphs, which we relabel as  $F_1, F_2, \ldots, F_{s-1}$. It is easy to verify that except for  differences $3k+2$ and $6k+4$, all other differences occur exactly four times (exactly one pink, one blue, and two opposite black with opposite directions) in $F_1, F_2, \ldots, F_{s-1}$. 

Next, we construct a $(C_2,C_4)$-subgraph $F_s$ that covers differences $3k+2$ and $6k+4$. Let $F_s=C_s\mathbin{\dot{\cup}} E_s$, where:
\begin{itemize}
\item $C_s= x_{0}\ x_{3k+2}\ x_{6k+4}\ x_{-(3k+2)}\ x_{0}$; and
\item $E_s= x_{1}\ x_{-(6k+3)}\ x_{1}$.
\end{itemize}
Color the edges of $C_s$ as follows: $x_{3k+2}\ x_{6k+4}$ blue, $x_{-(3k+2)}\ x_{0}$ pink, and the other two edges black and orient them away from the pink edge and towards the blue edge.
We see that $C_s$ covers the difference $3k+2$ exactly four times (one pink, one blue, and two black arcs), and $E_s$ covers $6k+4$ twice. 
Observe that the subgraphs $F_1,F_2,\ldots, F_{s}$ satisfy Conditions (B1) to (B3) of Lemma \ref{lem:4Kn-d}. Hence, $4K_n^{\bullet}$ admits an HOP $(C_2,C_4)$-decomposition.
\end{proof}


\begin{lem}{\label{lem:m=6-Samll2}}
There exists an HOP $(C_3,C_3)$-decomposition of $4K_9^{\bullet}$. 
\end{lem}
\begin{proof}  
Let $s=6$. By Lemma \ref{lem:4Kn-e} with $m_1=3$ and $m_2=3$, it suffices to find $(C_3,C_3)$-subgraphs $F_1,F_2,\ldots, F_6$ in $4K_n^{\bullet}$ that satisfy Conditions (D1) to (D3). Observe that such  $F_1,F_2,\ldots, F_6$ subgraphs are shown in Figure \ref{fig:C3D}.  Therefore, $4K_9^{\bullet}$ admits an HOP $(C_3,C_3)$-decomposition.
\end{proof}

\begin{figure}[p]
 \centering
   \includegraphics[width=\linewidth,height=\textheight,keepaspectratio]{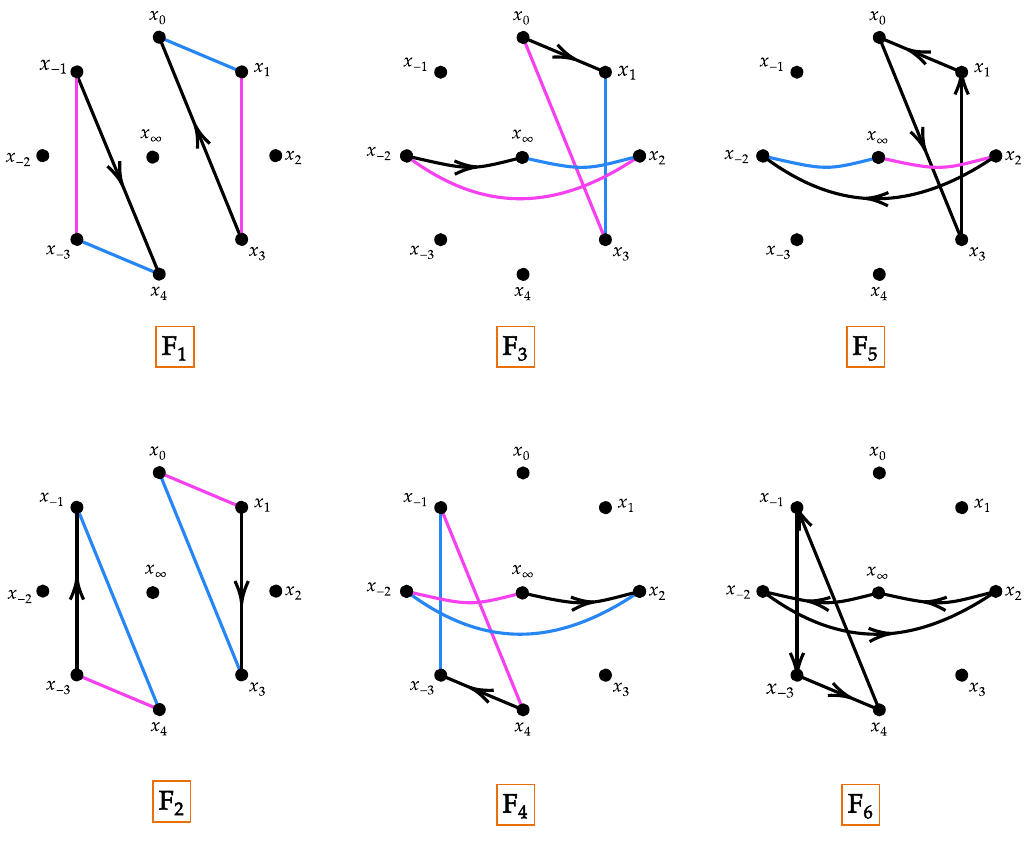}
    \caption{ $(C_3, C_3)$-subgraphs  of $4K_{9}^{\bullet}$ that generate  an HOP $(C_3,C_3)$-decomposition for $4K_{9}^{\bullet}$.}
    \label{fig:C3D}
  \end{figure}
\subsection{HOP Decompositions for $n\equiv 1\ ({\rm mod}\ 16)$ }
\begin{lem}{\label{lem:m=8-Samll0}}
Let $n\equiv 1\ ({\rm mod}\ 16)$. Then $4K_n^{\bullet}$ admits an HOP $(C_2, C_2, C_2, C_2)$-decomposition.  
\end{lem}
\begin{proof}
By Theorem \ref{theo:Cm-dec}, there exists a $(C_8)$-decomposition of $K_n$. Applying Corollary~\ref{lem:Newtool-01},  we see that  $4K_n^{\bullet}$ admits an HOP $(C_2, C_2, C_2, C_2)$-decomposition.
\end{proof}

\begin{lem}{\label{lem:m=8-Samll1}}
Let $n=16k+1$ for any integer $k\geq 1$. There exists an HOP $(C_2, C_3, C_3)$-decomposition of $4K_n^{\bullet}$. 
\end{lem}
\begin{proof}  
Let $V(K_n)=\{x_i: i\in \ZZ_n\}$, and let  $\rho=(x_0\ x_1\ \ldots \ x_{n-1})$ be a permutation on $V(K_n)$.
By the given labeling, each edge in $E(K_n)$ is of the form $x_ix_{i+d}$, where $d\in \{1,2,\ldots, 8k\}$. Therefore, the group $\langle \rho \rangle$ acting on $E(K_n)$ has $8k$ orbits, each of size $n$.

Our goal is to construct $k$ subgraphs of $K_n$, namely $F_1,F_2, \ldots, F_k$, where each $F_i$ is a disjoint union of a $(C_3, C_3)$-subgraph and two edge-disjoint $1$-regular subgraphs of order $2$ such that $F_1,F_2, \ldots, F_k$ jointly contain exactly one edge of each difference in $\{1,2,\ldots, 8k\}$.  Therefore, $\D=\{\rho^i(F_1), \rho^i(F_2), \ldots, \rho^i(F_k): i\in \ZZ_n\}$ is a decomposition for $K_n$ that satisfies the condition of  Lemma \ref{lem:Newtool-00} (for $\alpha=1$ and $m_1=m_2=3$); hence, by Lemma \ref{lem:Newtool-00}, the multigraph $4K_n^{\bullet}$ admits an HOP $(C_2, C_3, C_3)$-decomposition.

Now, for $i\in\{1,2,\ldots,2k\}$, construct the following closed walk:
$$ C_i= x_{1-i}\ x_{-(8k-3i+1)}\ x_{i}\ x_{1-i}.$$
It can be seen that $C_i$ is a cycle of length three, and traverses edges of differences 
$$8k-4i+2, 8k-2i+1, 2i-1.$$ 
Observe that $C_1, C_2, \ldots, C_{2k}$ jointly contain exactly one edge of each difference in $\{1,2,\ldots, 8k\}$  except for differences $4i$, for $i\in\{1,2,\ldots,2k\}$. 
The following subgraphs of $K_n$ (isomorphic to $K_2$) cover the leftover differences:
\begin{itemize}
\item $E_i=x_{2k+1}\ \ x_{2k+4i+1}$, for $i\in\{1,2,\ldots,k\}$;
\item $E_i=x_{2k+1}\ \ x_{-(14k-4i)}$, for $i\in\{k+1,\ldots,2k\}$.
\end{itemize}
Observe that none of the vertices in the interval \([x_{2k+1}, x_{8k+2}]\) are used in any \( C_i \) for \( i \in \{1, 2, \ldots, k\} \), and only vertices from this interval are used to construct \( E_i \) for \( i \in \{1, 2, \ldots, k\} \).
 Furthermore, none of the vertices in the interval \([x_{2k+1}, x_{11k+2}]\) are used in any $C_{i}$ for $i\in\{k+1,\ldots,2k\}$, and only vertices from this interval are used to construct $E_i$ for $i\in\{k+1,\ldots,2k\}$. Hence, for $i\in\{1,2,\ldots,2k\}$, the subgraph $C_{i}$ is disjoint from the subgraph $E_{i}$. Now, for $i\in\{1,2,\ldots,k\}$, let $F_i= C_{2i-1}\mathbin{\dot{\cup}} C_{2i} \mathbin{\dot{\cup}} (E_{2i-1}\cup E_{2i})$. It is easy to verify that $F_1, F_2, \dots, F_k$ satisfy the required conditions.
 \end{proof} 
\subsection{HOP Decompositions for $n\equiv 1 \text{ \normalfont{or }} 9\ ({\rm mod}\ 18)$ }
\begin{lem}{\label{lem:m=9-Samll0}}
Let $n=18k+1$ for any integer $k\geq 1$. There exists an HOP $(C_2, C_3, C_4)$-decomposition of $4K_n^{\bullet}$. 
\end{lem}
\begin{proof}  
Let $V(K_n)=\{x_i: i\in \ZZ_n\}$, and let  $\rho=(x_0\ x_1\ \ldots \ x_{n-1})$.
Each edge in $E(K_n)$ is of the form $x_ix_{i+d}$, where $d\in \{1,2,\ldots, 9k\}$. Therefore, the group $\langle \rho \rangle$ has $9k$ orbits on $E(K_n)$, each of size $n$.

We construct subgraphs $F_1, \ldots, F_k$, each a disjoint union of a $(C_3, C_4)$-subgraph and two edge-disjoint 1-regular subgraphs of order 2. If these subgraphs jointly contain exactly one edge of each difference in $\{1, \ldots, 9k\}$, then $\D = \{\rho^i(F_j) : i \in \ZZ_n, 1 \le j \le k\}$ is a decomposition for $K_n$ that satisfies the condition of  Lemma \ref{lem:Newtool-00} (for $\alpha=1$,  $m_1=3$, and $m_2=4$); hence, the multigraph $4K_n^{\bullet}$ admits an HOP $(C_2, C_3, C_4)$-decomposition.

Now, for $i\in\{1,2,\ldots,k\}$, construct the following closed walks:
$$ C_i= x_{0}\ x_{4i-3}\ x_{-1}\ x_{4i-1}\ x_{0};$$
$$ C_i'= x_{-2}\ x_{4k+2i-2}\ x_{-(4k+2i+1)}\ x_{-2}.$$
It can be seen that $C_i$ and $C_i'$ are two disjoint cycles of lengths four and three, respectively, and  traverse edges of differences 
$$C_i: 4i-3, 4i-2, 4i, 4i-1;$$ 
$$C_i': 4k+2i, d, 4k+2i-1.$$ 
In \( C_i' \), we have \( d = 8k + 4i - 1 \) for \( i \in \{1, 2, \ldots, \lfloor \tfrac{k+1}{4} \rfloor \} \), and \( d = 10k - 4i + 2 \) for \( i \in \{\lceil \tfrac{k+1}{4} \rceil, \ldots, k\} \).

Observe that, for $i\in\{1,2,\ldots,k\}$, $C_i$ and $C_i'$ jointly contain exactly one edge of each difference in $\{1,2,\ldots, 9k\}$  except  for the following differences:
\begin{enumerate}[\bf(i)]
\item $6k+2i-1$, for $i\in\{1,2,\ldots,k\}$; 
\item $6k+4i$, for $i\in\{1,2,\ldots,\floor{\frac{3k}{4}}\}$; 
\item $12k-4i+1$, for $i\in\{\floor{\frac{3k}{4}}+1,\ldots,k\}$. 
\end{enumerate}
The following subgraphs of $K_n$ cover the differences listed above exactly once:
\begin{enumerate}[\bf(i)]
\item $E_i=x_{-3}\ \ x_{-(6k+2i+2)}$, for $i\in\{1,2,\ldots,k\}$;
\item $E_i'=x_{-3}\ \ x_{-(6k+4i+3)}$, for $i\in\{1,2,\ldots,\floor{\frac{3k}{4}}\}$; 
\item $E_i'=x_{-3}\ \ x_{12k-4i-2}$, for $i\in\{\floor{\frac{3k}{4}}+1,\ldots,k\}$. 
\end{enumerate}
Notice that the vertex \( x_{-3} \), as well as all vertices in the interval \([x_{6k-1}, x_{12k-1}]\), are not used in any of the cycles \( C_i \) or \( C_i' \), and that only vertices from this interval are used to construct the graphs \( E_i \) and \( E_i' \). Hence, for \( i \in \{1, 2, \ldots, k\} \), the cycles \( C_i \) and \( C_i' \) are disjoint from the edge-disjoint subgraphs \( E_i \) and \( E_i' \).

Now, for $i\in\{1,2,\ldots,k\}$, let $F_i= C_{i}\mathbin{\dot{\cup}} C_{i}' \mathbin{\dot{\cup}} (E_{i}\cup E_{i}')$. It is easy to verify that $F_1, F_2, \dots, F_k$ satisfy the required conditions.
\end{proof} 
\begin{lem}{\label{lem:m=9-Samll1}}
Let $n=18k+9$ for any integer $k\geq 0$. There exists an HOP $(C_2, C_3,C_4)$-decomposition of $4K_n^{\bullet}$. 
\end{lem}
\begin{proof}  
Notice that  $n=18k+9=9(2k+1)$, and we have $K_{9(2k+1)}=H_1\oplus H_2$, where $H_1$ is a disjoint union of $2k+1$ copies of $K_{9}$, and $H_2$ is a complete equipartite graph with $2k+1$ parts of size $9$. In Lemma ~\ref{lems-1starter}(i) \cite{LDMSaj}, it is proved that $4K_{9}^{\bullet}$ admits an HOP $(C_2, C_3,C_4)$-factorization. Consequently, by Lemma \ref{Gtool3}, $4H_1^{\bullet}$ admits an HOP  $(C_2, C_3,C_4)$-decomposition. Hence,  we may assume $k\geq 1$ and it suffices to show that $4K_{(2k+1)[9]}^{\bullet}$ admits an HOP $(C_2, C_3,C_4)$-decomposition.

The complete equipartite graph $K_{(2k+1)[9]}$ can be viewed as a circulant graph, Circ$(n; S)$, with the vertex set $\{x_i: i\in \ZZ_n\}$, and edge set $\{x_ix_{i+d}: i\in \ZZ_n, d\in S\}$, where 
$$S=\{1,2,\ldots, 2k, 2k+2, \ldots,4k+1, 4k+3,\ldots,6k+2, 6k+4, \ldots,8k+3, 8k+5, \ldots,9k+4 \}.$$
Let  $\rho=(x_0\ x_1\ \ldots \ x_{n-1})$ be a permutation on $V(K_{(2k+1)[9]})$.
Observe that $S$ contains  all differences from $1$ to $9k+4$, except for the differences  $2k+1, 4k+2, 6k+3, 8k+4$.
 Therefore, the group $\langle \rho \rangle$ acting on $E(K_{(2k+1)[9]})$ has $9k$ orbits, each of size $n=18k+9$.

First, we construct subgraphs $F_1,F_2, \ldots, F_k$, where each $F_i$ is a disjoint union of a $(C_3,C_4)$-subgraph  and two edge-disjoint 1-regular subgraphs of order 2, such that $F_1,F_2, \ldots, F_k$ jointly contain exactly one edge of each difference in $S$.

Now, for $i\in\{1,2,\ldots,k\}$, construct the following closed walks:
$$ C_i= x_{0}\ x_{2i-1}\ x_{-1}\ x_{4k+2i+1}\ x_{0};$$
$$ C_i'= x_{-2}\ x_{2k+i-1}\ x_{-(6k+i+5)}\ x_{-2}.$$
It can be seen that $C_i$ and $ C_i'$ are  disjoint cycles of lengths four and three, respectively.
The 4-cycle $C_i$  traverses edges of differences 
$$C_i: 2i-1, 2i, 4k+2i+2, 4k+2i+1.$$ 
%
The 3-cycle $C_i'$  traverses edges of differences 
$$C_i': 2k+i+1, d, 6k+i+3.$$ 
Notice that, for $i\in\{1,2,\ldots,\floor{\frac{k}{2}}\}$, the difference $d=8k+2i+4$, and for $i\in\{\floor{\frac{k}{2}}+1\ldots,k\}$, the difference $d=10k-2i+5$.

Observe that the subgraphs $C_i$ and $C_i'$, for $i\in\{1,2,\ldots,k\}$, jointly contain exactly one edge of each difference in $S$  except for the following differences:
\begin{enumerate}[(i)]
\item $4k-i+2$, for $i\in\{1,2,\ldots,k\}$; and
\item $8k-i+4$, for $i\in\{1,2,\ldots,k\}$.  
\end{enumerate}
Note that there are $2k$ differences from $S$ that are not covered by any $C_i$ or $C_i'$. The following subgraphs of $K_n$ (isomorphic to $K_2$) cover these leftover differences:
\begin{itemize}
\item $E_i=x_{-3}\ \ x_{-(4k-i+5)}$, for $i\in\{1,2,\ldots,k\}$; and
\item $E_i'=x_{-(4k+4)}\ \ x_{4k-i}$, for $i\in\{1,2,\ldots,k\}$.
\end{itemize}
Notice that none of the vertices in the intervals \([x_{3k}, x_{4k+2}]\) and \([x_{-(6k+5)}, x_{-3}]\) are used in any \(C_i\) or \(C_i'\), and only vertices in these intervals are used to construct the subgraphs \(E_i\) and \(E_i'\). Hence, \(C_i\) and \(C_i'\) are disjoint from \(E_i\) and \(E_i'\).

For \(i \in \{1, 2, \ldots, k\}\), let \(F_i = C_i \mathbin{\dot{\cup}} C_i' \mathbin{\dot{\cup}} (E_i \cup E_i')\). It is easy to verify that $F_1, F_2, \dots, F_k$ satisfy the required conditions.
\end{proof} 
\subsection{HOP Decompositions for $n\equiv 1 \text{ \normalfont{or }}5\ ({\rm mod}\ 20)$ }
\begin{lem}{\label{lem:m=10-Samll0}}
Let $n\equiv 1 \text{ \normalfont{or }} 5\ ({\rm mod}\ 20)$. Then the following decompositions of $4K_n^{\bullet}$ exist:
\begin{enumerate}[{\bf(i)}]
\item {\label{lem:m=10-Samll0-1}} an HOP $(C_2, C_2, C_2, C_2, C_2)$-decomposition;
\item {\label{lem:m=10-Samll0-2}} an HOP $(C_2, C_2, C_2, C_4)$-decomposition;
\item {\label{lem:m=10-Samll0-3}} an HOP $(C_2, C_2, C_3, C_3)$-decomposition; and
\item {\label{lem:m=10-Samll0-4}} an HOP $(C_2, C_2, C_6)$-decomposition.
\end{enumerate} 
\end{lem}
\begin{proof}
By Theorems \ref{theo:Cm-dec} and \ref{theo:(Cm1,...,Cmt)-dec}, we know there exist a $(C_{10})$-decomposition, a $(C_6, C_4)$-decomposition, a $(C_4, C_3, C_3)$-decomposition, and a $(C_4, C_6)$-decomposition of $K_n$, and by Corollary~\ref{lem:Newtool-01}, these decompositions give rise  to  HOP decompositions listed in (i)-(iv), respectively.
\end{proof}
%
%
\begin{lem}{\label{lem:m=10-Samll1}}
Let $n=20k+1$ for any integer $k\geq 1$. There exists an HOP $(C_2, C_4, C_4)$-decomposition of $4K_n^{\bullet}$. 
\end{lem}
\begin{proof}  
Let \( V(K_n) = \{x_i : i \in \mathbb{Z}_n\} \), and let \( \rho = (x_0\ x_1\ \ldots\ x_{n-1}) \).
Each edge in \( E(K_n) \) is of the form \( x_i x_{i+d} \), where \( d \in \{1, 2, \ldots, 10k\} \); thus,  \( \langle \rho \rangle \) has \( 10k \) orbits, each of size \( n \).

We construct subgraphs \( F_1, F_2, \ldots, F_k \), where each \( F_i \) is a disjoint union of a \( (C_4, C_4) \)-subgraph and two edge-disjoint 1-regular subgraphs of order 2, such that \( F_1, F_2, \ldots, F_k \) jointly contain exactly one edge of each difference in \( \{1, 2, \ldots, 10k\} \). Therefore, \( \mathcal{D} = \{\rho^i(F_1), \rho^i(F_2), \ldots, \allowbreak \rho^i(F_k) : i \in \mathbb{Z}_n\} \) is a decomposition of \( K_n \) that satisfies the condition of Lemma~\ref{lem:Newtool-00} (for \( \alpha = 1 \), \( m_1 = 4 \), and \( m_2 = 4 \)); hence,  the multigraph \( 4K_n^{\bullet} \) admits an HOP \( (C_2, C_4, C_4) \)-decomposition.

Now, for $i\in\{1,2,\ldots,2k\}$, construct the following closed walk:
$$ C_i= x_{1-i}\ x_{i}\ x_{-(2k+i)}\ x_{4k+i}\ x_{1-i}.$$
It can be seen that $C_i$ is a cycle of length four, and  traverses edges of differences 
$$C_i: 2i-1, 2k+2i, 6k+2i, 4k+2i-1.$$ 
%
%
%
Observe that the $C_i$, for $i\in\{1,2,\ldots,2k\}$, jointly contain exactly one edge of each difference in $\{1,2,\ldots, 10k\}$  except for differences $2i$ and $8k+2i-1$, for $i\in\{1,2,\ldots,k\}$.
The following subgraphs cover these leftover differences:
\begin{enumerate}[\bf(i)]
\item $E_i=x_{6k+1}\ \ x_{6k+2i+1}$, for $i\in\{1,2,\ldots,k\}$;
\item $E_i'=x_{6k+1}\ \ x_{-(6k-2i+1)}$, for $i\in\{1,2,\ldots,k\}$.
\end{enumerate}
Notice that none of the vertices in the interval \([x_{6k+1}, x_{16k}]\) are used in any \(C_i\), and only vertices in this interval are used to construct the subgraphs \(E_i\) and \(E_i'\). Hence, for \(i \in \{1, 2, \ldots, k\}\), cycles \(C_{2i-1}\) and \(C_{2i}\) are disjoint from the subgraphs \(E_i\) and \(E_i'\).

Now, for $i\in\{1,2,\ldots,k\}$, let $F_i= C_{2i-1}\mathbin{\dot{\cup}} C_{2i} \mathbin{\dot{\cup}} (E_{i} \cup E_{i}')$. It is easy to verify that $F_1, F_2, \dots, F_k$ satisfy the required conditions.
\end{proof} 


\begin{lem}{\label{lem:m=10-Samll2}}
Let $n=20k+1$ for any integer $k\geq 1$. There exists an HOP $(C_2, C_3, C_5)$-decomposition of $4K_n^{\bullet}$. 
\end{lem}
\begin{proof}  
Let $V(K_n)=\{x_i: i\in \ZZ_n\}$, and let  $\rho=(x_0\ x_1\ \ldots \ x_{n-1})$.
Each edge in $E(K_n)$ is of the form $x_ix_{i+d}$, where $d\in \{1,2,\ldots, 10k\}$; thus, $\langle \rho \rangle$ has $10k$ orbits, each of size $n$.

We construct subgraphs $F_1,F_2, \ldots, F_k$, where each $F_i$ is a disjoint union of a $(C_3, C_5)$-subgraph and two edge-disjoint $1$-regular subgraphs of order $2$, and $F_1,F_2, \ldots, F_k$ jointly contain exactly one edge of each difference in $\{1,2,\ldots, 10k\}$.  Therefore, $\D=\{\rho^i(F_1), \rho^i(F_2), \ldots,\allowbreak \rho^i(F_k): i\in \ZZ_n\}$ is a decomposition for $K_n$ that satisfies the condition of  Lemma \ref{lem:Newtool-00} (for $\alpha=1$,  $m_1=3$, and $m_2=5$); hence, the multigraph $4K_n^{\bullet}$ admits an HOP $(C_2, C_3, C_5)$-decomposition.

Now, for $i\in\{1,2,\ldots,k\}$, construct the following closed walks:
$$ C_i= x_{0}\ x_{2i-1}\ x_{-1}\ x_{2k+i-1}\ x_{-(3k+i)}\ x_{0};$$
$$ C_i'= x_{-(k+i)}\ x_{3k}\ x_{-(8k-i+2)}\ x_{-(k+i)}.$$
It can be seen that $C_i$ and $C_i'$ are disjoint cycles of length  five and three, respectively.
The 5-cycle $C_i$  traverses edges of differences: 
$$C_i: 2i-1, 2i, 2k+i, 5k+2i-1, 3k+i.$$ 
%
%
The 3-cycle $C_i'$  traverses edges of differences: 
$$C_i': 4k+i, 9k+i-1, 7k-2i+2.$$ 
%
Observe that, for $i\in\{1,2,\ldots,k\}$, $C_i$ and $C_i'$ jointly contain exactly one edge of each difference $\{1,2,\ldots, 10k\}$  except for differences $10k$ and  $7k+i$, for $i\in\{1,2,\ldots,2k-1\}$. 
The following subgraphs cover these leftover differences:
\begin{itemize}
\item $E_i=x_{-5k}\ \ x_{8k-i+1}$, for $i\in\{1,2,\ldots,2k-1\}$, and 
\item $E_{2k}=x_{-5k}\ \ x_{5k}$.
\end{itemize}
Notice that the vertex \(x_{-5k}\), as well as none of the vertices in the interval \([x_{3k+1}, x_{12k-1}]\), are used in any \(C_i\) or \(C_i'\), and only vertices in this interval are used to construct the subgraphs \(E_i\). Hence, for \(i \in \{1, 2, \ldots, k\}\), cycles \(C_i\) and \(C_i'\) are disjoint from the subgraphs \(E_{2i-1}\) and \(E_{2i}\).

Now, for \(i \in \{1, 2, \ldots, k\}\), let \(F_i = C_i \mathbin{\dot{\cup}} C_i' \mathbin{\dot{\cup}} (E_{2i-1} \cup E_{2i})\). It is easy to verify that $F_1, F_2, \dots, F_k$ satisfy the required conditions.
\end{proof} 

\begin{lem}{\label{lem:m=10-Samll3}}
Let $n=20k+5$ for any integer $k\geq 1$. There exists an HOP $(C_2, C_8)$-decomposition of $4K_n^{\bullet}$. 
\end{lem}
\begin{proof}  
Let $s=4k+1$. By Lemma \ref{lem:2Kn-d} with $m_1=2$ and $m_2=8$, it suffices to find $(C_2,C_8)$-subgraphs $F_1,F_2,\ldots, F_s$ in $2K_n^{\circ}$ that satisfy Conditions (A1) to (A3).

As in Lemma \ref{lem:2Kn-d}, let $V(2K_n^{\circ})=\{x_i: i\in \ZZ_{n-1} \}\cup \{x_{\infty}\}$.  Let $\rho_{\circ}$ be the permutation on $E(2K_n^{\circ})$ that preserves the color of the edges, and  is induced by  the permutation  $\rho=(x_{\infty})(x_0 \ x_1\  x_2 \ \ldots \ x_{n-2})$. Observe that the group $\langle \rho_{\circ} \rangle$ has the following orbits on the edge set of $2K_n^{\circ}$:
\begin{itemize}
\item for each $d\in \{1,2,\ldots, 10k+1\}$, we have a pink and a black orbit $\{x_ix_{i+d}: i\in \ZZ_{n-1} \}$;
\item a pink and a black orbit $\{ x_ix_{i+(10k+2)}: i=0,1,\ldots,\frac{n-3}{2}\}$; and
\item a pink and a black orbit  $\{x_ix_{\infty}: i\in \ZZ_{n-1} \}$.
\end{itemize}
Let \( S = \{1, 2, \ldots, 10k + 2\} \) be the set of differences. Note that \( 10k + 2 \) is the diameter difference.
First, for $i\in\{1,2,\ldots, 4k+1\}$, construct the following closed walk:
$$ C_i= x_{0}\ x_{i}\ x_{-(5k+1)}\ x_{10k+2-i}\ x_{10k+2}\ x_{-(10k+2-i)}\ x_{5k+1}\ x_{-i}\ x_{0}.$$
It can be seen that $C_i$ is a cycle of length eight, and traverses edges of differences $i$ and $5k+i+1$ four times. Since these differences come from an orbit of size \( n - 1 \), the four edges occur as two pairs of the form \( \{e, \rho^{\tfrac{n-1}{2}}_\circ(e)\} \).

The 8-cycles $C_1, C_2,\ldots, C_{4k+1}$ jointly cover each of the following differences exactly four times:
$$1,2,\ldots, 4k+1, 5k+2, 5k+3,\ldots, 9k+1, 9k+2.$$
The differences that are not covered by any $C_i$ are  $9k+2+i$, $4k+1+i$, for $i\in\{1,2,\ldots, k\}$, and difference $\infty$. 
Next,  we construct cycles of length two that jointly cover each of the leftover differences exactly four times, except for the diameter difference $10k+2$, which should be covered exactly twice.
\begin{itemize}
\item The following 2-cycles cover the differences $9k+2+j$, for $j\in\{1,2,\ldots, k-1\}$, and the difference $\infty$ exactly four times. 
\begin{itemize}
\item  $E_j=x_{-(4k+1)}\ x_{5k+1+j}\ x_{-(4k+1)}$, for $j\in\{1,2,\ldots, k-1\}$, and  $E_k=x_{-(3k+2)}\ x_{\infty}\ x_{-(3k+2)}$. 
\item  $E_j'=x_{6k+1}\ x_{-(5k+1-j)}\ x_{6k+1}$, for $j\in\{1,2,\ldots, k-1\}$, and  $E_k'=x_{7k}\ x_{\infty}\ x_{7k}$. 
\end{itemize}
Notice that \(E_j' = \rho_{\circ}^{\frac{n-1}{2}}(E_j)\). Moreover, none of the vertices in the intervals \([x_{2k+1}, x_{5k}]\), \([x_{5k+2}, x_{8k+1}]\), \([x_{-(8k+1)}, x_{-(5k+2)}]\), and \([x_{-5k}, x_{-(2k+1)}]\) are used in any \(C_i\), for \(i \in \{1, 2, \allowbreak \ldots, 2k\}\), and only vertices in these intervals are used to construct the subgraphs \(E_j\) and \(E_j'\). Hence, for \(j \in \{1, 2, \ldots, k\}\) and \(i \in \{1, 2, \ldots, 2k\}\), the subgraphs \(E_j\), \(E_j'\), and \(C_i\) are pairwise disjoint. 

Relabel \(E_j\) and \(E_j'\) jointly as \(E_i\), for \(i \in \{1, 2, \ldots, 2k\}\). Now, for \(i \in \{1, 2, \ldots, 2k\}\), let \(F_i = C_i \mathbin{\dot{\cup}} E_i\). We see that \(F_1, F_2, \ldots, F_{2k}\) are \((C_2, C_8)\)-subgraphs of \(2K_n\).

\item The following 2-cycles cover difference the $4k+1+j$, for $j\in\{1,2,\ldots, k\}$, exactly four times. 
\begin{itemize}
\item  $E_j''=x_{-(4k+1)}\ x_{j}\ x_{-(4k+1)}$, for $j\in\{1,2,\ldots, k\}$.
\item  $E_j'''=x_{6k+1}\ x_{-(10k+2-j)}\ x_{6k+1}$, for $j\in\{1,2,\ldots, k\}$.
\end{itemize}
First, notice that \(E_j''' = \rho_{\circ}^{\frac{n-1}{2}}(E_j'')\). Moreover, none of the vertices in the intervals \([x_{1}, x_{2k}]\), \([x_{5k+2}, x_{6k+1}]\), \([x_{-(10k+1)}, x_{-(8k+2)}]\), and \([x_{-5k}, x_{-(4k+1)}]\) are used in any \(C_i\), for \(i \in \{2k+1, 2k+2, \ldots, 4k\}\), and only vertices in these intervals are used to construct the subgraphs \(E_j''\) and \(E_j'''\). Hence, the subgraphs \(E_j''\), \(E_j'''\), and \(C_i\) are pairwise disjoint.

Next, relabel \(E_j''\) and \(E_j'''\) jointly as \(E_i\), for \(i \in \{2k+1, 2k+2, \ldots, 4k\}\). Now, for \(i \in \{2k+1, 2k+2, \ldots, 4k\}\), let \(F_i = C_i \mathbin{\dot{\cup}} E_i\). We see that \(F_{2k+1}, F_{2k+2}, \ldots, F_{4k}\) are \((C_2, C_8)\)-subgraphs of \(2K_n\).

\item The only difference that is not covered yet is the diameter difference \(10k + 2\). Let \(E_{4k+1} = x_{1}\ x_{-(10k+1)}\ x_{1}\), and observe that \(E_{4k+1}\) is disjoint from \(C_{4k+1}\). Now, let \(F_{4k+1} = C_{4k+1} \mathbin{\dot{\cup}} E_{4k+1}\). We see that \(F_{4k+1}\) is a \((C_2, C_8)\)-subgraph of \(2K_n\).

\end{itemize}

Next,  color the edges of the cycles in the $(C_2, C_8)$-subgraphs alternately pink and black. Observe that the subgraphs $F_1,F_2,\ldots, F_s$ satisfy Conditions (A1) to (A3) of Lemma \ref{lem:2Kn-d}. 
\end{proof}

\begin{lem}{\label{lem:m=10-Samll4}}
Let $n=20k+5$ for any integer $k\geq 1$. There exists an HOP $(C_2, C_4, C_4)$-decomposition of $4K_n^{\bullet}$. 
\end{lem}
\begin{proof}  
Observe that $2n=10(4k+1)$; let $s=4k+1$. By Lemma \ref{lem:2Kn-d} with $m_1=2, m_2=m_3=4$, it suffices to find $(C_2, C_4, C_4)$-subgraphs $F_1,F_2,\ldots, F_s$ in $2K_n^{\circ}$ that satisfy Conditions (A1) to (A3).
The group $\langle \rho_{\circ} \rangle$ has the same orbits as those described in Lemma~\ref{lem:m=10-Samll3}.

For $k=1$, see Figure \ref{fig:s4}. For $k\geq 2$ and $i\in\{1,2,\ldots,4k+1\}$, construct the following two closed walks. 
$$ C_i= x_{1}\ x_{i+1}\ x_{-(10k+1)}\ x_{-(10k+1-i)}\ x_{1};$$
$$ C_i'= x_{-1}\ x_{-(i+1)}\ x_{10k+1}\ x_{10k+1-i}\ x_{-1}.$$
Notice that for \(i \in \{1, 2, \ldots, 4k+1\}\), \(C_i\) and \(C_i'\) are disjoint cycles of length four and traverse edges of differences \(i, 10k-i+2, i, 10k- i +2\). Observe that edges of the same difference occur in diametrical pairs. Hence, \(C_i\) and \(C_i'\), for \(i \in \{1, 2, \ldots, 4k+1\}\), jointly cover each of the following differences exactly four times:
$$1,2,\ldots, 4k+1, 6k+1, 6k+2,\ldots, 10k, 10k+1.$$
The differences not covered by any \(C_i\) or \(C_i'\) are $10k+2$,  $\infty$, and $4k+1+i$, for $i\in\{1,2,\ldots, 2k-1\}$. 
Next,  we construct cycles of length two that cover the leftover differences exactly four times, except for the diameter difference $10k+2$, which should be covered exactly twice. 
\begin{itemize}
\item The following 2-cycles cover the differences \(4k + 1 + j\), for \(j \in \{1, 2, \ldots, 2k - 1\}\), as well as the difference \(\infty\), exactly four times.
\begin{itemize}
\item  $E_j=x_{0}\ x_{-(4k+1+j)}\ x_{0}$, for $j\in\{1,2,\ldots, 2k-1\}$ , and  $E_{2k}=x_{0}\ x_{\infty}\ x_{0}$. 
\item  $E_j'=x_{10k+2}\ x_{6k+1-j}\ x_{10k+2}$, for $j\in\{1,2,\ldots, 2k-1\}$ , and  $E_{2k}'=x_{10k+2}\ x_{\infty}\ x_{10k+2}$. 
\end{itemize}
Notice that \(E_j' = \rho_{\circ}^{\frac{n-1}{2}}(E_j)\). Moreover, the vertices \(x_0\), \(x_{10k+2}\), and the vertices in the intervals \([x_{4k+2}, x_{6k}]\) and \([x_{-(6k)}, x_{-(4k+2)}]\) are not used in any \(C_i\) or \(C_i'\), for \(i \in \{1, 2, \ldots, 4k\}\), and only vertices in these intervals are used to construct graphs \(E_j\) and \(E_j'\). Hence, the subgraphs \(E_j\), \(E_j'\), \(C_i\), and \(C_i'\) are pairwise disjoint. 

Next, relabel \(E_j\) and \(E_j'\) jointly as \(E_i\), for \(i \in \{1, 2, \ldots, 4k\}\). Now, let \(F_i = C_i \mathbin{\dot{\cup}} C_i' \mathbin{\dot{\cup}} E_i\). We see that \(F_1, F_2, \ldots, F_{4k}\) are \((C_2, C_4, C_4)\)-subgraphs of \(2K_n\).

\item The only difference that is not covered yet is $10k+2$. Let $E_{4k+1}= x_{0}\ x_{10k+2}\ x_{0}$, and observe $E_{4k+1}$ is disjoint from $C_{4k+1}$ and $C'_{4k+1}$. Now, let $F_{4k+1}=C_{4k+1}\mathbin{\dot{\cup}} C_{4k+1}'\mathbin{\dot{\cup}} E_{4k+1}$. We see that $F_{4k+1}$ is a $(C_2, C_4, C_4)$-subgraph  of $2K_n$. 
\end{itemize}
Next, in each \((C_2, C_4, C_4)\)-subgraph, color the edges of the 2-cycle pink and black. Color the edges of one of the 4-cycles, say \(C_i\), alternately pink and black. Note that edges of the same length receive the same color. Then, color the edges of the other 4-cycle, \(C_i'\), in the opposite way of  \(C_i\) (see Figure~\ref{fig:s4}).

The \((C_2, C_4, C_4)\)-subgraphs \(F_1, F_2, \ldots, F_s\) satisfy Conditions (A1) to (A3) of Lemma~\ref{lem:2Kn-d}. 
\end{proof}  

\begin{figure}[p]
 \centering
   \includegraphics[width=\linewidth,height=0.95\textheight,keepaspectratio]{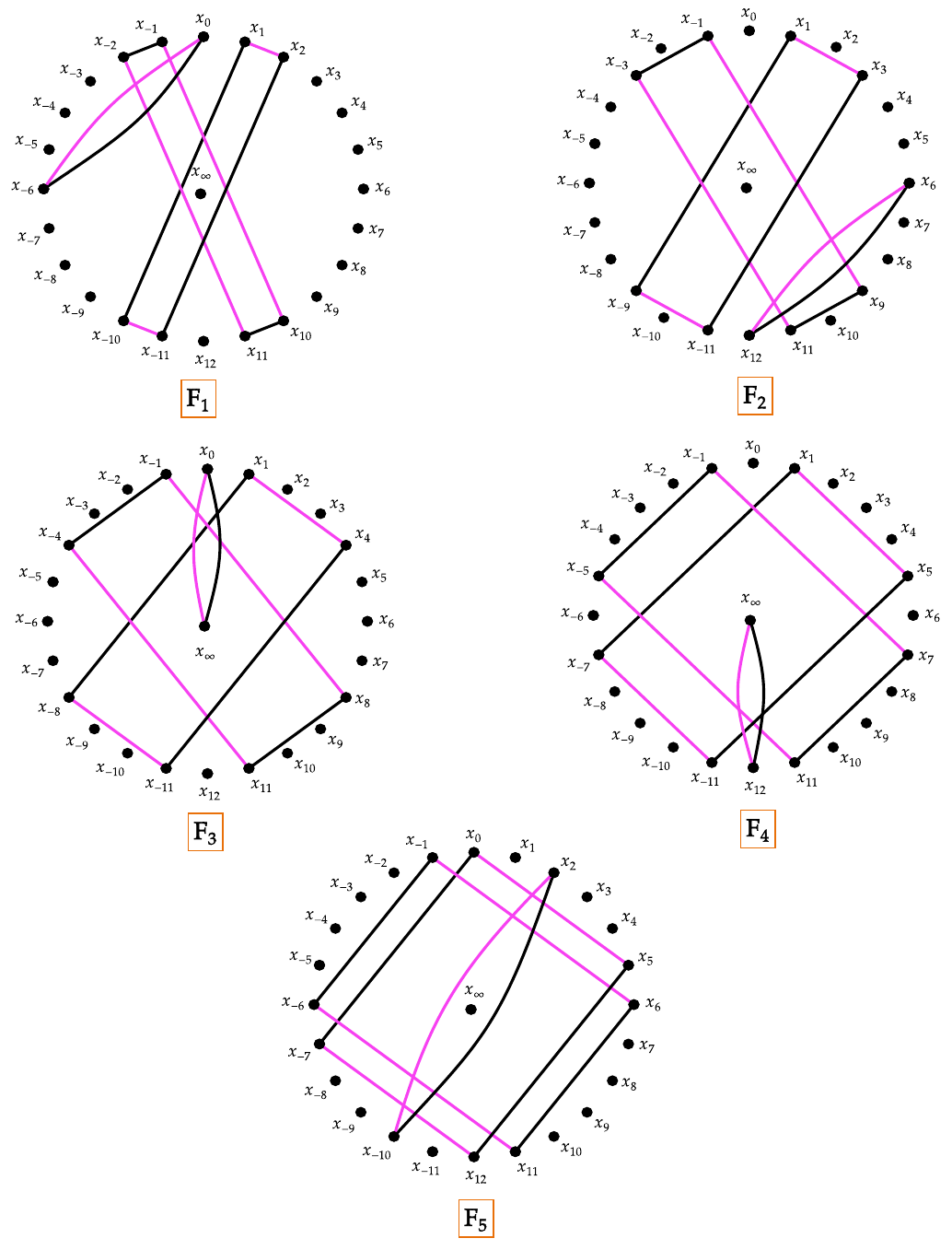}
    \caption{ $(C_2, C_4, C_4)$-subgraphs  of $2K_{25}^{\circ}$ that generate an HOP $(C_2, C_4, C_4)$-decomposition of $4K_{25}^{\bullet}$.}
     \label{fig:s4}
  \end{figure}
  

\begin{lem}{\label{lem:m=10-Samll5}}
Let $n=20k+5$ for any integer $k\geq 1$. There exists an HOP $(C_2, C_3, C_5)$-decomposition of $4K_n^{\bullet}$. 
\end{lem}
\begin{proof}
Let $s=4k+1$. By Lemma \ref{lem:4Kn-d} with $m_1=2$, $m_2=3$, and $m_3=5$, it suffices to find $(C_2, C_3, C_5)$-subgraphs $F_1,F_2,\ldots, F_s$ in $4K_n^{\bullet}$ that satisfy Conditions (B1) to (B3).

As in Lemma \ref{lem:4Kn-d}, let $V(4K_n^{\bullet})=\{x_i: i\in \ZZ_{n-1} \}\cup \{x_{\infty}\}$.  Let $\rho_{\bullet}$ be the permutation on $E(4K_n^{\bullet})$ that preserves the color (and orientation) of the edges, and is induced by  the permutation  $\rho=(x_{\infty})(x_0 \ x_1\  x_2 \ \ldots \ x_{n-2})$. Observe that the group $\langle \rho_{\bullet} \rangle$ has the following orbits on the edge set of $4K_n^{\bullet}$:
\begin{itemize}
\item for each $d\in \{1,2,\ldots, 10k+1\}$, we have a pink and a blue orbit $\{x_ix_{i+d}: i\in \ZZ_{n-1} \}$;
\item for each $d\in \{1,2,\ldots, 20k+3\}$, we have a black orbit $\{(x_i, x_{i+d}): i\in \ZZ_{n-1} \}$;
\item a pink and a blue orbit $\{ x_ix_{i+(10k+2)}: i=0,1,\ldots,\frac{n-3}{2}\}$;
\item a pink and a blue orbit  $\{x_ix_{\infty}: i\in \ZZ_{n-1} \}$; and
\item black orbits $\{(x_{\infty}, x_{i}): i\in \ZZ_{n-1} \}$ and $\{(x_i, x_{\infty}): i\in \ZZ_{n-1} \}$.
\end{itemize}

First, we construct the 5-cycles and the 3-cycles of the $(C_2, C_3, C_5)$-subgraphs. 
For $k=1$ see Figure {\ref{fig:6-F4}}. For $k\geq2$ and $i\in\{1,2,\ldots, k\}$,  construct the following closed walks:
$$ C_i= x_{0}\ x_{2i-1}\ x_{-3}\ x_{2k+i-1}\ x_{-(3k+i+2)}\ x_{0}$$
$$ C_i'= x_{-(k+i+2)}\ x_{3k}\ x_{-(8k-i+6)}\ x_{-(k+i+2)}$$
It can be seen that $C_i$ and $C_i'$ are disjoint cycles of length five and three, respectively. 
\begin{figure}[p]
 \centering
   \includegraphics[width=\linewidth,height=0.95\textheight,keepaspectratio]{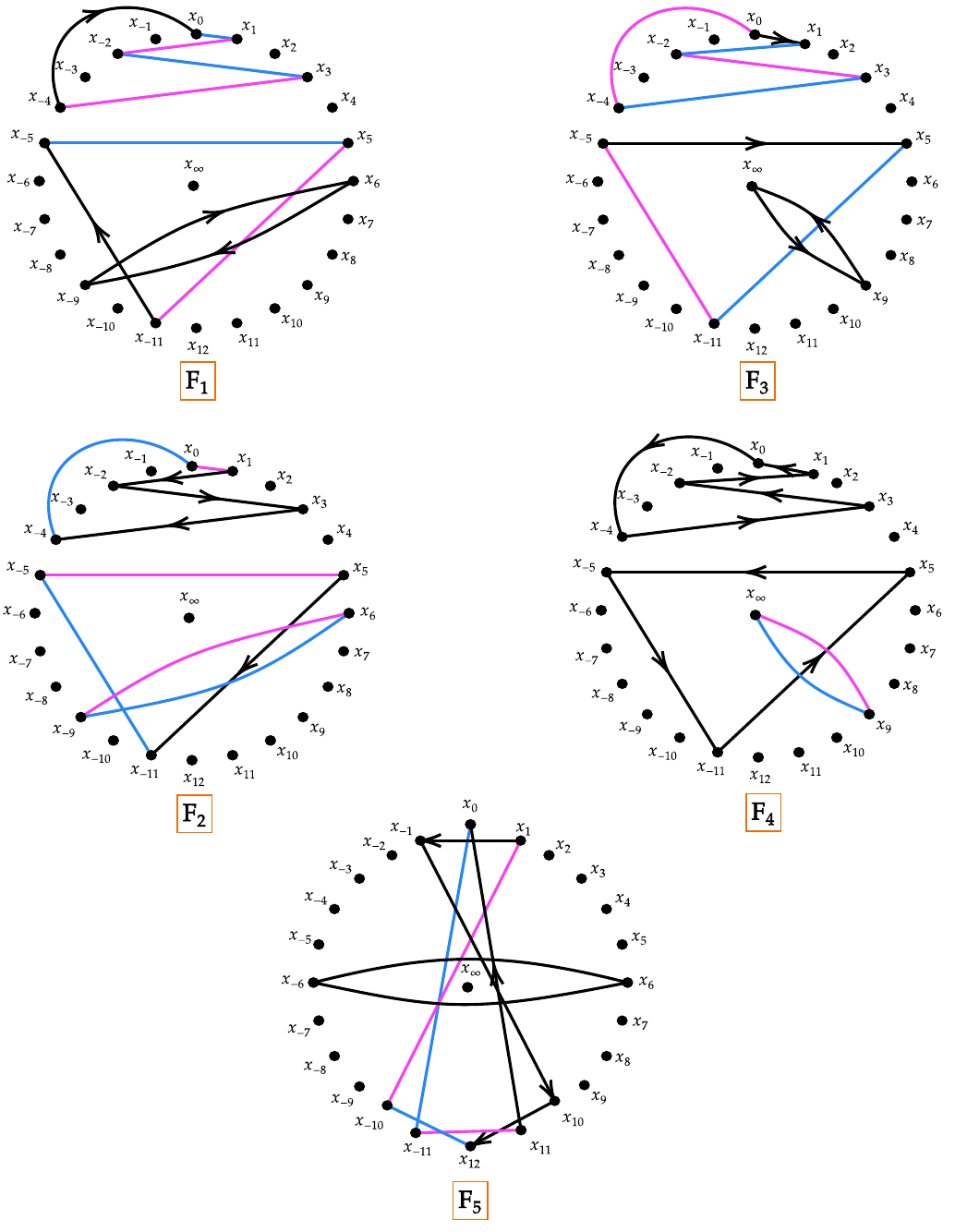}
    \caption{ $(C_2, C_3, C_5)$-subgraphs  of $4K_{25}^{\bullet}$ that generate  an HOP $(C_2, C_3, C_5)$-decomposition of $4K_{25}^{\bullet}$.}
    \label{fig:6-F4}
  \end{figure}
The 5-cycles \(C_i\) traverse edges of difference
\[
2i - 1, 2i + 2, 2k + i + 2, 5k + 2i + 1, 3k + i + 2.
\]
%
%
%
%
%
%
The 3-cycles $C_i'$  traverse edges of difference  
$$C_i': 4k+i+2, 9k+i-2, 7k-2i+4.$$ 
%
%
%
%
%
%
The differences that are not covered by any $C_i$ or $C_i'$ are $2, 2k+1, 10k-1, 10k, 10k+1, 10k+2, \infty$, and $7k+i+2$, for $i\in\{1,2,\ldots,2k-4\}$. 
The following $C_2$-subgraphs of $K_n$ cover these leftover differences, exactly once:
\begin{enumerate}[\bf(i)]
\item $E_i= x_{-(3k+2)}\ x_{4k+i}\ x_{-(3k+2)}$, for $i\in\{1,2,\ldots,2k-4\}$; 
\item $E_{2k-3}= x_{3k+1}\ x_{5k+2}\ x_{3k+1}$;
\item $E_{2k-2}= x_{-(3k+2)}\ x_{7k-3}\ x_{-(3k+2)}$;
\item $E_{2k-1}= x_{-(3k+2)}\ x_{7k-2}\ x_{-(3k+2)}$; and
\item $E_{2k}= x_{-(3k+2)}\ x_{\infty}\ x_{-(3k+2)}$.
\end{enumerate}
Note that neither the vertex \(x_{-(3k+2)}\) nor any of the vertices in the interval \([x_{3k+1}, x_{12k-2}]\) are used by any \(C_i\) or \(C_i'\), and only vertices in this interval are used to construct the 2-cycles \(E_i\). Hence, all \(C_i\) and \(C_i'\) are disjoint from the \(E_i\).

For each \(i \in \{1,2,\ldots,k\}\), let \(C_i^{(1)}, C_i^{(2)}, C_i^{(3)}, C_i^{(4)}\) be four copies of \(C_i\), and let \(C_i'^{(1)}, C_i'^{(2)}, \allowbreak C_i'^{(3)}, C_i'^{(4)}\) be four copies of \(C_i'\). For each \(i \in \{1,2,\ldots,2k\}\), let \(E_i^{(1)}\) and \(E_i^{(2)}\) be two copies of \(E_i\). Now, apply the proof of Lemma~\ref{lem:Gtool4} to color the copies of each cycle so that they satisfy Condition {\bf (C1)} of Definition~\ref{def}, and so that
\begin{itemize}
\item \(C_i^{(1)} \oplus C_i^{(2)} \oplus C_i^{(3)} \oplus C_i^{(4)} = 4C_i^{\bullet}\),
\item \(C_i'^{(1)} \oplus C_i'^{(2)} \oplus C_i'^{(3)} \oplus C_i'^{(4)} = 4C_i'^{\bullet}\), and
\item \(E_i^{(1)} \oplus E_i^{(2)} = 4E_i^{\bullet}\).
\end{itemize}
Now, let $T_i^{(1)},T_i^{(2)},T_i^{(3)},T_i^{(4)}$ be the following subgraphs:
 $$T_i^{(1)}= C_i^{(1)}\cup C_i'^{(1)}\cup E_i^{(1)};$$
 $$T_i^{(2)}= C_i^{(2)}\cup  C_i'^{(2)}\cup E_{k+i}^{(1)};$$
 $$T_i^{(3)}= C_i^{(3)}\cup C_i'^{(3)}\cup  E_i^{(2)};$$
 $$T_i^{(4)}= C_i^{(4)}\cup C_i'^{(4)}\cup E_{k+i}^{(2)}.$$
It is easy to see that  $T_i^{(1)},T_i^{(2)},T_i^{(3)},T_i^{(4)}$ are all $(C_2, C_3, C_5)$-subgraphs of $4K_n^{\bullet}$ that satisfy Condition {\bf (C1)} of Definition \ref{def}. We apply the above procedure for each $i\in\{1,2,\ldots,k\}$,  so we have $4k$ $(C_2, C_3, C_5)$-subgraphs,  which we relabel as  $F_1, F_2, \ldots, F_{s-1}$. 

Observe that except for  differences $2$, $10k+1$, and $10k+2$, all other differences occur exactly four times (exactly one pink, one blue, and two black with opposite directions) in $F_1, F_2, \ldots, F_{s-1}$. 

Next, we construct a $(C_2, C_3, C_5)$-subgraph $F_s$ that covers differences $2$, $10k+1$, and $10k+2$. Let $F_s=C_s\mathbin{\dot{\cup}} C_s'\mathbin{\dot{\cup}} E_s$, where:
\begin{itemize}
\item $C_s= x_{1}\ x_{-1}\ x_{10k}\ x_{10k+2}\ x_{-10k}\ x_{1}$; 
\item $C_s'= x_{0}\ x_{10k+1}\  x_{-(10k+1)}\ x_{0}$; and
\item $E_s= x_{-(5k+1)}\ x_{5k+1}\ x_{-(5k+1)}$.
\end{itemize}
Color the edges of $C_s$ as follows: $x_{10k+2}\ x_{-10k}$ blue, $x_{-10k}\ x_{1}$ pink, and $x_{1}\ x_{-1}$, $\ x_{-1}\ x_{10k}$, $x_{10k}\ x_{10k+2}$ black with orientation away from the pink edge and towards the blue edge. 
Color the edges of $C_s'$ as follows: $x_{10k+1}\  x_{-(10k+1)}$ pink, $x_{-(10k+1)}\ x_{0}$ blue, and $x_{0}\ x_{10k+1}$ black with orientation away from the pink edge and towards the blue edge (see Figure \ref{fig:6-F4}). 
Notice that $C_s$ and $C_s'$ jointly cover differences $10k+1$ and $2$ exactly four times (one pink edge, one blue edge, and two opposite  black arcs), and $E_s$ covers difference $10k+2$ twice. 
Observe that the subgraphs $F_1,F_2,\ldots, F_{s}$ satisfy Conditions (B1) to (B3) of Lemma \ref{lem:4Kn-d}. Hence, $4K_n^{\bullet}$ admits an HOP $(C_2, C_3, C_5)$-decomposition.
\end{proof}

\section{Conclusion}

In this paper, we proved in Theorem~\ref{theo-new-1} that $\mathrm{HOP}(2^{\langle s \rangle}, 2m_1, 2m_2)$ has a solution whenever $n \equiv 1 \pmod{(2m_1 + 2m_2)}$ or $n \equiv m_1 + m_2 \pmod{(2m_1 + 2m_2)}$. In Theorem~\ref{theo:HOP(Cm1,...,Cmt)-dec}, we also considered the case where $m = m_1 + m_2 + \dots + m_t \leq 10$, and proved that $\mathrm{HOP}(2^{\langle s \rangle}, 2m_1, \dots, 2m_t)$ has a solution whenever $n = s + m$ is odd and $n(n - 1) \equiv 0 \pmod{2m}$.

Observe that in Theorems~\ref{theo-new-1} and~\ref{theo:HOP(Cm1,...,Cmt)-dec} we had specific conditions on $n$, but the obvious necessary conditions are more general. So it is natural to ask whether these results hold more generally. Hence, we propose the following problems:

\vspace{0.2cm}

\noindent {\bf{Problem 1.}}
Let $s \geq 0$, and let $2 \leq m_1 \leq m_2$ be integers. 
Assume $m = m_1 + m_2$. 
Is it true that $\mathrm{HOP}(2^{\langle s \rangle}, 2m_1, 2m_2)$ has a solution if and only if 
$2n(n - 1) \equiv 0 \pmod{m}$?

\vspace{0.25cm}

\noindent {\bf{Problem 2.}} 
Let $s \geq 0$, and let $2 \leq m_1 \leq \dots \leq m_t$ be integers. 
Assume $m = m_1 + \dots + m_t \leq 10$. 
Is it true that $\mathrm{HOP}(2^{\langle s \rangle}, 2m_1, \dots, 2m_t)$ has a solution if and only if 
$2n(n - 1) \equiv 0 \pmod{m}$?

\section{Acknowledgments}

The author would like to thank her PhD supervisor, Dr. Mateja \v{S}ajna, for her invaluable guidance and support during this research.

\clearpage

\end{document}